\documentclass[11pt, reqno]{amsart}
\usepackage{a4wide}
\usepackage{wrapfig}
\usepackage{amsmath}
\usepackage{amsthm}
\usepackage{mathrsfs}
\usepackage{color}
\usepackage{xcolor}
\usepackage{graphicx}
\usepackage{amssymb}
\usepackage{ dsfont }
\usepackage{url}
\usepackage{multicol}
\usepackage{tikz}
\usepackage{amsmath}
\usepackage{fancyhdr}
\usetikzlibrary{matrix}
\usetikzlibrary{arrows}
\usepackage{subfig}
\usepackage{hyperref}
\usepackage{listings}

\usepackage{esint}

\usepackage{enumitem}
\usepackage{varwidth}

\allowdisplaybreaks

\title[Asymptotic stability of peakons]{Asymptotic stability of peakons for the Novikov equation}

\author[J.M. Palacios]{Jos\'e Manuel Palacios}
\address{Institut Denis Poisson, Universit\'e de Tours, Universit\'e d'Orleans, CNRS, Parc Grandmont 37200, Tours, France}
\email{jose.palacios@etu.univ-tours.fr}

\newcommand{\be}{\begin{equation}}
\newcommand{\ee}{\end{equation}}
\newcommand{\bp}{\begin{proof}}
\newcommand{\ep}{\end{proof}}
\newcommand{\bel}{\begin{equation}\label}
\newcommand{\eeq}{\end{equation}}
\newcommand{\bea}{\begin{eqnarray}}
\newcommand{\eea}{\end{eqnarray}}
\newcommand{\bee}{\begin{eqnarray*}}
\newcommand{\eee}{\end{eqnarray*}}
\newcommand{\ben}{\begin{enumerate}}
\newcommand{\een}{\end{enumerate}}

\newcommand{\R}{\mathbb{R}}

\newcommand{\N}{\mathbb{N}}

\newcommand{\supp}{\operatorname{supp}}

\newcommand{\sgn}{\operatorname{sgn}}

\newtheorem{thm}{Theorem}[section]
\newtheorem{cor}[thm]{Corollary}
\newtheorem{lem}[thm]{Lemma}
\newtheorem{prop}[thm]{Proposition}

\newtheorem{defn}[thm]{Definition}

\theoremstyle{remark}
\newtheorem{rem}{Remark}[section]

\definecolor{codegreen}{rgb}{0,0.6,0}
\definecolor{codegray}{rgb}{0.5,0.5,0.5}
\definecolor{codepurple}{rgb}{0.58,0,0.82}
\definecolor{backcolour}{rgb}{0.95,0.95,0.92}

\lstdefinestyle{mystyle}{
	backgroundcolor=\color{backcolour},   
	commentstyle=\color{codegreen},
	keywordstyle=\color{magenta},
	numberstyle=\tiny\color{codegray},
	stringstyle=\color{codepurple},
	basicstyle=\footnotesize,
	breakatwhitespace=false,         
	breaklines=true,                 
	captionpos=b,                    
	keepspaces=true,                 
	numbers=left,                    
	numbersep=5pt,                  
	showspaces=false,                
	showstringspaces=false,
	showtabs=false,                  
	tabsize=2
}
\numberwithin{equation}{section}

\usepackage{pgfplots}
\pgfplotsset{compat=newest}



\theoremstyle{definition}

\numberwithin{ej}{section}

\begin{document}





\renewcommand{\sectionmark}[1]{\markright{\thesection.\ #1}}
\renewcommand{\headrulewidth}{0.5pt}
\renewcommand{\footrulewidth}{0.5pt}
\begin{abstract}
The Novikov equation is an integrable Camassa-Holm type equation with a cubic nonlinearity. One of the most important features of this equation is the existence of peaked traveling waves, also called peakons. This paper aims to prove the asymptotic stability of peakon solutions under $H^1(\R)$-perturbations satisfying that their associated momentum density defines a non-negative Radon measure. Motivated by Molinet's work \cite{Mo,Mo2,Mo3}, we shall first prove a Liouville property for $H^1(\R)$ global solutions belonging to a certain class of \emph{almost localized} functions. More precisely, we show that such solutions have to be a peakon. The main novelty in our analysis in comparison to the Camassa-Holm equation comes from the fact that in our present case the momentum is not a conserved quantity and may be unbounded along the trajectory. In this regard, to prove the Liouville property, we used a new Lyapunov functional not related to the (not conserved) momentum of the equation.
\end{abstract}

\maketitle 

\section{Introduction}

\subsection{The model} This paper is concerned with the Novikov equation \begin{align}\label{novikov_eq}
u_t-u_{txx}+4u^2u_x=3uu_xu_{xx}+u^2u_{xxx}, \qquad (t,x)\in\R^2,
\end{align}
where $u(t,x)$ is a real-valued function. This equation was derived by Novikov \cite{No} in a symmetry classification of nonlocal partial differential equations with cubic nonlinearity. By using the perturbative symmetry approach \cite{MiNo}, which yields necessary conditions for a PDE to admit infinitely many symmetries, Novikov  was able to isolate equation \eqref{novikov_eq} and derive its first few symmetries. Later, he was able to find an associated scalar Lax-pair, proving the integrability of the equation. On the other hand, Hone and Wang recently found a matrix Lax-pair representation of the Novikov equation, specifically, they showed that \eqref{novikov_eq} arises as a zero curvature equation $F_t-G_x+[F,G]=0$ which is the compatibility condition for the linear system \cite{HoWa} \begin{align*}
\Psi_x=F(y,\lambda)\Psi \quad \hbox{and} \quad \Psi_t=G(y,\lambda)\Psi,
\end{align*}
where $y=u-u_{xx}$ and the matrices $F$ and $G$ are defined by \begin{align*}
F=\left(\begin{matrix}
0 & \lambda y & 1 \\ 0 & 0 & \lambda y \\ 1 & 0 & 0
\end{matrix}\right), \quad G= \left(\begin{matrix}
\tfrac{1}{3\lambda^2}-uu_x & \tfrac{1}{\lambda}u_x-\lambda u^2y & u_x^2 \\ \tfrac{1}{\lambda}u & -\tfrac{2}{3\lambda^2} & -\tfrac{1}{\lambda}u_x-\lambda u^2y 
\\ -u^2 & \tfrac{1}{\lambda}u & \tfrac{1}{3\lambda^2}+uu_x
\end{matrix}\right).
\end{align*}
Moreover, by using this matrix Lax-pair representation, Hone and Wang showed how the Novikov equation is related by a reciprocal transformation to a negative flow in the Sawada-Kotera hierarchy. 

\medskip

The Novikov equation possesses infinitely many conservation laws, among which, the most important ones are given by
\begin{align}\label{cons_e}
E(u):=\int_\R\left(u^2(t,x)+u_x^2(t,x)\right)dx \quad \hbox{and}\quad F(u):=\int \Big(u^4+2u^2u_x^2-\dfrac{1}{3}u_x^4\Big)dx.
\end{align}
Solutions of \eqref{novikov_eq} are known to satisfy several symmetry properties: shifts in space and time, i.e. the mapping $u(t,x)\mapsto u(t+t_0,x+x_0)$ among solutions to \eqref{novikov_eq} is preserved, as well as space-time invertion, i.e. if $u(t,x)$ is a solution of \eqref{novikov_eq}, then $u(-t,-x)$ is another solution.

\medskip

One of the most important features of the Novikov equations is the existence of \emph{peakon} and \emph{antipeakon} solutions \cite{HoWa} which are peaked traveling waves with a discontinuous derivative at the crest. They are explicitly given by \[
\pm\varphi_{c}(x-ct)=\pm\sqrt{c}\varphi(x-ct):=\pm\sqrt{c}e^{-\vert x-ct\vert}, \qquad c>0.
\]
Moreover, the Novikov equation also exhibit multi-peakons solutions. More precisely, for any given natural number $n\in\N$, let us denote by $\vec{q}=(q_1,...,q_n)$ and $\vec{p}=(p_1,...,p_n)$ the position and momenta vectors. Then, the $n$-peaked traveling wave solution on the line is given by $
u(t,x)=\sum_{i=1}^n p_i(t)\exp(-\vert x-q_i(t)\vert)$, where $p_i$ and $q_i$ satisfy the following system of $2n$-differential equations
\begin{align}\label{multipeak}
\begin{cases}
\dfrac{dq_i}{dt}=u^2(q_i)=\displaystyle\sum_{j,k=1}^np_jp_ke^{-\vert q_i-q_j\vert-\vert q_i-q_k\vert},
\\ \displaystyle\dfrac{dp_i}{dt}=-p_iu(q_i)u_x(q_i)=p_i\sum_{j,k=1}^np_jp_k\sgn(q_i-q_j)e^{-\vert q_i-q_j\vert-\vert q_i-q_k\vert}.\end{cases}
\end{align}
There exists some similar expressions for periodic peakons and multipeakon solutions but we do not intend to deepen in this direction. On the other hand, equation \eqref{novikov_eq} can be rewritten in a compact form in terms of its \emph{momentum density} as
\begin{align}\label{nov_eq_y}
y_t+u^2y_x+3uu_xy=0, \quad \hbox{where} \quad y:=u-u_{xx},
\end{align}
which can be regarded as a cubic nonlinear generalization of the celebrated Camassa-Holm (CH) equation \cite{CH,FuFo}, \begin{align}\label{CH}
u_t-u_{txx}=uu_{xxx}+2u_xu_{xx}-3uu_x\quad \hbox{equivalently} \quad y_t+uy_x+2u_xy=0,
\end{align}
or the Degasperis-Procesi (DP) equation \cite{DP}, \begin{align}\label{DP}
u_t-u_{txx}=uu_{xxx}+3u_xu_{xx}-4uu_x \quad \hbox{equivalently} \quad y_t+uy_x+3u_xy=0.
\end{align}
It is worth noticing that the last three equations in terms of the momentum densities correspond to transport equations for $y(t)$. As a consequence, initial data with signed initial momentum density give rise to solutions with the same property. This is one of the key points to prove that smooth and decaying initial data with signed initial momentum density give rise to global solutions.

\medskip

Regarding the CH and the DP equations, both can be derived as a model for the propagation of unidirectional shallow water waves over a flat bottom by writing the Green-Naghdi equations in Lie-Poisson Hamiltonian form and then making an asymptotic expansion which keeps the Hamiltonian structure \cite{AlLa,CH,CoLa,Jo}. Moreover, both of them can be written in Hamiltonian form \[
\partial_tE'(u)=-\partial_xF'(u),
\]
where for the Camassa-Holm equation $E(u)$ and $F(u)$ are given by \[
E(u):=\int u^2+u_x^2 \quad \hbox{and} \quad F(u):=\int u^3+uu_x^2
\]
while for the Degasperis-Procesi equation they are given by \[
E(u):=\int yv=\int 5v^2+4v_x^2+v_{xx}^2 \quad \hbox{and}\quad F(u):=\int u^3,
\]
where $v:=(4-\partial_x^2)^{-1}u$. Moreover, both of them belongs to the so-called $b$-family introduced by Degasperis, Holm and Hones in \cite{DeHoHo}, \[
u_t-u_{txx}=bu_xu_{xx}+uu_{xxx}-(b+1)uu_x.
\]
In \cite{MiNo} it was shown that the $b$-family corresponds to an integrable equation only when $b=2,3$, which corresponds exactly to the CH and the DP equations respectively.

\medskip

On the other hand, the Novikov equation, as well as the CH and the DP equations, can also be written in a nonlocal form in the following way. From now on we shall denote by $p(x)$ the fundamental solution of $1-\partial_x^2$ in $\R$, that is $p:=\tfrac{1}{2}e^{-\vert x\vert}$. Then, we can rewrite \eqref{novikov_eq} as 
\begin{align}\label{nov_eq_2}
u_t+u^2u_x=-p*\left(3uu_xu_{xx}+2u_x^3+3u^2u_x\right),
\end{align}
which can be understood as a nonlocal perturbation of Burgers-type equations \[
u_t+\tfrac{1}{3}(u^3)_x=0,
\]
or more generally as a nonlinear nonlocal transport equation. This latter fact has many implications, for instance, from the blow-up criteria for transport equations we obtain that singularities are caused by the focusing of characteristics.

\medskip

At this point it is clear that the Novikov equation shares many of its remarkable analytic properties with both the CH and the DP equations, as the existence of a Lax-pair, the completely integrability and the bi-Hamiltonian structure \cite{DP,HoWa}, but also all of them exhibit both existence of peaked traveling waves as well as the phenomenon of wave breaking \cite{CH,ChGuLiQu,CoLa,DP,No}. This latter one means that the wave profile remains bounded while its slope becomes unbounded. As the authors explain in \cite{ChGuLiQu}, understanding the wave-breaking mechanism not only presents fundamental importance from a mathematical point of view but also a great physical interest since it would help to provide a key-mechanism for localizing energy in conservative systems by
forming one or several small-scale spots. Finally, we remark that, unlike the Novikov equation, peakon solutions for the CH and the DP equations have a slightly different form, which is given by \[
\widetilde{\varphi}_c(x-ct)=c\varphi(x-ct):=ce^{-\vert x-x_0-ct\vert}, \qquad c\in\R\setminus\{0\},\  x_0\in\R.
\]
It is worth noticing that in sharp contrast with the Novikov equation, CH and DP peakons can move in both directions, left and right, just by changing the sign of $c$, while all Novikov peakons and anti-peakons move to the right.

\medskip

About the stability of these peaked solitary waves, the first proof of orbital stability was given in the Camassa-Holm case for $H^1$-perturbations assuming that their associated momentum density defines a non-negative Radon measure \cite{CoMo2}. The orbital stability for perturbations in the whole energy space $H^1(\R)$ was proved by Constantin and Strauss in \cite{CoSt} (see also \cite{LiLi} for a proof in the Degasperis-Procesi case). Later, following the ideas in \cite{CoSt,LiLi} Liu et al. proved the orbital stability for peakons in the Novikov case under the additional assumption of non-negative initial momentum density \cite{LiLiQu}.

\medskip
 
From a physical point of view, these peakons, as well as the ones for the Novikov equation, reveal some similarities to the well-known Stokes waves of greatest height, i.e. traveling waves of maximum possible amplitude that are solutions to the governing equations for irrotational water waves \cite{Co,To}. These traveling waves (Stokes waves) are smooth everywhere except at the crest, where the lateral tangents differ.


%
%

\subsection{Initial data space}

Before stating our results we need to introduce some functional spaces and notation. Following the ideas of \cite{CoMo,EM1,EM2,Mo} we define  
\[
Y:=\big\{u\in H^1(\R): \ u-u_{xx}\in\mathcal{M}_b\big\},
\]
where $\mathcal{M}_b$ denotes the space of Radon measures  with finite total mass on $\R$. Moreover, from now on we shall denote by $Y_+$ the subspace defined by $Y_+:=\{u\in Y: \ u-u_{xx}\in\mathcal{M}_{b}^+\}$, where $\mathcal{M}_{b}^+$ denotes the space of non-negative finite Radon measures on $\R$. A crucial remark in what follows is that, for any function $v\in C_0^\infty(\R)$ we have
\begin{align}\label{positive_mom_1}
v(x)&=\dfrac{1}{2}\int_{-\infty}^x e^{x'-x}(v-v_{xx})(x')dx'+\dfrac{1}{2}\int_x^\infty e^{x-x'}(v-v_{xx})(x')dx'
\end{align}
and \begin{align}\label{positive_mom_2}
v_x(x)&=-\dfrac{1}{2}\int_{-\infty}^x e^{x'-x}(v-v_{xx})(x')dx'+\dfrac{1}{2}\int_x^\infty e^{x-x'}(v-v_{xx})(x')dx'
\end{align}
Therefore, if $v-v_{xx}\geq 0$ on $\R$ we conclude that $\vert v_x\vert\leq v$. Thus, by density of $C_0^\infty(\R)$ in $Y$, we deduce the same property for functions $v\in Y_+$.

\begin{rem}
We recall the following standard estimate which shall be useful in the sequel:
\[
\Vert u\Vert_{W^{1,1}}=\Vert p*(u- u_{xx})\Vert_{W^{1,1}}\lesssim \Vert u-u_{xx}\Vert_{\mathcal{M}},
\]
and hence it also holds that \[
\Vert u_{xx}\Vert_{\mathcal{M}}\leq \Vert u\Vert_{L^1}+\Vert u-u_{xx}\Vert_{\mathcal{M}}.
\]
Thus, we have $
Y(\R)\hookrightarrow \left\{u\in W^{1,1}(\R): \, u_x\in \mathrm{BV}(\R)\right\}$,
where $\mathrm{BV}(\R)$ denotes the space of functions with bounded variation. 
\end{rem}

With all of these definitions at hand we are able to introduce the most important definition throughout this paper. 

\begin{defn}[$H^1$-almost localized solution]\label{def_AL}
We say that a solution $u\in C(\R, H^1(\R))$ of equation \eqref{nov_eq_2} satifying $u-u_{xx}\in C_{ti}(\R,\mathcal{M}_{b}^+)$ is $H^1$-almost localized if there existe a $C^1$-function $x(\cdot)$ such that the following holds: For any $\varepsilon>0$, there exists $R_\varepsilon>0$ such that for all $t\in\R$ we have  \begin{align}\label{Y_al_def}
\int_{\vert x\vert>R_\varepsilon} \left(u^2+u_x^2\right)(t,\cdot+x(t))dx\leq \varepsilon.
\end{align}
\end{defn}

\begin{rem}\label{remark1}
In \cite{Mo}-\cite{Mo2}, instead of using Definition \ref{def_AL}, the author used what he called $Y$-almost localization, i.e. he replaced the functional in \eqref{Y_al_def} by \begin{align}\label{CH_functional_al}
\int \left(u^2(t)+u_x^2(t)\right)\Phi(\cdot-x(t))dx+\left\langle u(t)-u_{xx}(t),\Phi(\cdot-x(t))\right\rangle\leq \varepsilon,
\end{align}
and \begin{align}\label{DP_functional_al}
\left\langle u(t)-u_{xx}(t),\Phi(\cdot-x(t))\right\rangle\leq \varepsilon,
\end{align}
for the Camassa-Holm and the $b$-family respectively, where $\Phi$ corresponds to any continuous function $0\leq \Phi\leq 1$ satisfying $\supp\Phi\subset [-R_\varepsilon,R_\varepsilon]^c$. This change is related to the fact that the CH equation conserve both, the energy and the momentum, while the $b$-family conserve the momentum. Nevertheless, in the case of the CH, DP and Novikov equations, since we can prove that $H^1$-almost localized solutions are uniformly exponentially decaying, all of these characterizations are actually equivalent (see \cite{Mo2} for the equivalence between \eqref{CH_functional_al} and \eqref{DP_functional_al}). 
\end{rem}

\subsection{Main results}

The following theorem is the main result of this paper and give us the asymptotic stability of peakon solutions for the Novikov equation.
\begin{thm}\label{MT1}
Let $c>0$ be fixed. There exists an universal constant $1\gg\varepsilon^\star>0$ such that for any $\beta\in(0,c)$ and any initial data $u_0\in Y_+$ satisfying 
\begin{align}\label{smallness_hip}
\Vert u_0-\varphi_c\Vert_{H^1}\leq \varepsilon^\star\Big(\tfrac{\beta}{c}\Big)^8,
\end{align}
the following property holds: There exists $c^*>0$ with $\vert c-c^*\vert\ll c$ and a $C^1$ function $x:\R\to\R$ satisfying $\dot{x}(t)\to c^*$ as $t\to+\infty$
\[
u(t,\cdot+x(t))\rightharpoonup \varphi_{c^*} \ \hbox{ in }\ H^1(\R).
\]
where\footnote{By this we mean $u\in C(\R,H^1(\R))$ with $y\in C_{ti}(\R,\mathcal{M}_b^+(\R))$.} $u\in C_{ti}(\R,Y_+)$ is the global weak solution to equation \eqref{nov_eq_2} associated to $u_0$. Moreover, for any $z\in\R$ the following strong convergence holds
\begin{align}\label{strong_h1_conv_peakon_mthm}
\lim_{t\to+\infty}\Vert u(t)-\varphi_{c^*}(\cdot-x(t))\Vert_{H^1((-\infty,z)\cup(\beta t,+\infty))}=0.
\end{align}
\end{thm}
The main ingredient in the proof of Theorem \ref{MT1} is a rigidity property of the Novikov equation. 

\begin{thm}\label{MT2}
Let us suppose that $u\in C(\R,H^1(\R))$ with $u-u_{xx}\in C_{ti}(\R,\mathcal{M}_b^+)$ is an $H^1$-almost localized solution of \eqref{novikov_eq} that is not identically zero. Then, there exists $c^*>0$ and $x_0\in\R$ such that \[
u(t)=\sqrt{c^*}\varphi(\cdot-x_0-c^*t), \qquad \forall t\in\R.
\]
\end{thm}

The main ingredients in the proof of Theorem \ref{MT2} are the almost monotonicity of the energy, the finite speed of propagation of the momentum density, the existence of a Lyapunov functional and some continuity results with respect to the initial data for the $H^1$-topology.

\begin{rem}
This theorem implies, in particular, that an $H^1$-almost localized solution with non-negative momentun density cannot be smooth for any time. More precisely, if $u\in C(\R, H^1(\R))$ with $u-u_{xx}\in C_{ti}(\R,\mathcal{M}_b^+)$ is a $H^1$-almost localized solution of the Novikov equation that belongs to $H^{3/2}(\R)$ for some $t\in\R$, then $u$ must to be the trivial solution.
\end{rem}

Our method of proof is certainly strongly motivated by the remarkable work of Molinet in \cite{Mo} for the Camassa-Holm case (see also \cite{Mo2,Mo3}). However, as we shall see, due to the lack of conservation of momentum, the Novikov equation presents several new difficulties that we shall have to address. For instance, given a solution $u(t)$, since the global well-posedness requires the momentum density to have finite total mass on $\R$, apriori we are not allowed to study global limit solutions associated to $u(t)$. Nevertheless, by using an almost monotonicity result for the $H^1$-norm at the right of some curves, we shall be able to prove that for solutions staying close enough to peakon's trajectory, the associated limit objects are uniformly exponentially decaying and belong to $Y_+$ for all times. 
Another new difficulty is that the Lyapunov functional in \cite{Mo} was related to the conservation of the momentum. Here we introduce a new Lyapunov functional that is simpler and seems to work for a wider class   of CH-type equations. We point out that this new proof gives a simplification of Molinet's approach for the rigidity result, which can be useful for several types of CH-equations with peakon solutions.
%
%

\medskip

It is important to point out that all of these results, as well as the ones obtained in \cite{Mo,Mo2}, come from a series of remarkable previous works in the context of KdV-type equations. The interested reader can consult \cite{MaMe1,MaMe2,MaMeTs} for these previous results. 

\begin{rem}
From now on we shall focus on the peakon case $\varphi_c$. Nevertheless, notice that by using the invariance $u(t,x)\mapsto -u(t,x)$ we also deduce the asymptotic stability of the antipeakon profile $-\varphi_c$ where $c>0$, with perturbations in the class of $H^1$ functions with momentum density belonging to $\mathcal{M}_b^-(\R)$. 
\end{rem}

\subsection{Organization of this paper}
This paper is organized as follow. In Section \ref{preliminaries} we introduce some  definitions and state a series of results needed in our analysis, for instance, the well-posedness result in the class of solutions we shall work with. In section \ref{sec_MT2} we prove the rigidity result for the Novikov equation. Finally, in section \ref{sec_MT1} we prove the asymptotic stability of peakon solutions.

\section{Preliminaries}\label{preliminaries}

\subsection{Preliminaries and definitions}

In the sequel we shall need the following family of functions. Let $\{\rho_n\}_{n\in\N}$ be a mollifiers family definied by \begin{align}\label{def_rho}
\rho_n(x):=n\left(\int_\R\rho(\xi)d\xi\right)^{-1}\rho(nx), \quad \hbox{ where } \quad \rho(x):=\begin{cases}
e^{\frac{1}{x^2-1}} & \hbox{for } \vert x\vert<1
\\ 0 & \hbox{for } \vert x\vert\geq 1.
\end{cases}
\end{align}
Notice that for any $n\in\N$ we have $\Vert \rho_n\Vert_{L^1}=1$. On the other hand, for any $p\in [1,\infty]$ and any $T>0$ we shall denote by $\Vert f\Vert_{L_T^pH^1_x}$ the norm given by \[
\Vert f\Vert_{L_T^pH^1_x}^p:=\int_{-T}^T\left(\int_\R \big(f^2+f_x^2\big)(t,x)dx\right)^{p/2}dt.
\]
From now on we shall also denote by $C_b(\R)$ the set of bounded continuous functions on $\R$, and by  $C_c(\R)$ the set of compactly supported continuous functions on $\R$. Throughout this paper we shall also need the following definitions.
\begin{defn}[Weakly convergence of measures]\label{def_weakly_conv}
We say that a sequence $\{\nu_n\}\subseteq\mathcal{M}$ converge weakly towards $\nu\in\mathcal{M}$, which we shall denote by $\nu_n\rightharpoonup \nu$, if \[
 \langle \nu_n,\phi\rangle\to\langle \nu,\phi\rangle, \quad \hbox{for any }\, \phi\in C_c(\R).
\]
\end{defn}

\begin{rem}\label{weak_weakstar_conv}
Notice that we are adopting the standard Measure Theory's notation for the \emph{weak convergence} of a measure. Nevertheless, we recall that from a Functional Analysis point of view this convergence corresponds to the weak-* convergence on Banach spaces.
\end{rem}

\begin{defn}[Tightly and weak continuity of measure-valued functions] Let $I\subseteq \R$ be an interval.
\begin{enumerate}
\item We say that a function $f\in C_{ti}(I,\mathcal{M}_b)$ if for any $\phi\in C_b(\R)$ the map $t\mapsto\langle f(t)\phi\rangle$ is continuous on $I$.
\item We say that a function $f\in C_{w}(I,\mathcal{M})$ if for any $\phi\in C_c(\R)$ the map $t\mapsto\langle f(t)\phi\rangle$ is continuous in $I$.
\end{enumerate}
\end{defn}

\begin{defn}[Weak convergence in $C_{ti}(I)$] Let $I\subseteq \R$ be an interval. We say that a sequence $f_n\rightharpoonup f$ in $C_{ti}(I,\mathcal{M}_b)$ if for any $\phi \in C_b(\R)$ we have \[
\langle f_n(\cdot)\phi\rangle \to \langle f(\cdot)\phi\rangle \,\hbox{ in } C(I).
\]
\end{defn}
%

Let us finish this section by recalling a standard Measure Theory lemma 
(see for instance \cite{Ma}, Theorem $1.24$). 
%
\begin{lem}\label{lem_measures_2}
Let $\Omega$ be any locally compact metric space. Let us consider $\{\mu_n\}_{n\in\N}\subset\mathcal{M}_+(\Omega)$ a sequence of Radon measures weakly converging to some $\mu\in\mathcal{M}_+(\Omega)$ (see Definition \ref{def_weakly_conv}). Then, for every open set $V\subset\Omega$ we have \[
\mu(V)\leq\liminf_{n\to+\infty}\mu_n(V).
\]
\end{lem}
This weak-lower semicontinuity property shall be useful in our proof and shall enable us to approximate the momentum density associated to solutions of equation  \eqref{nov_eq_2} by smooth solutions and pass to the limit.

\subsection{Well-posedness}

In the proofs of Theorems \ref{MT1} and \ref{MT2} we shall need to approximate non-smooth solutions of equation \eqref{nov_eq_2} by sequences of smooth solutions. In this regard, we shall need a global well-posedness result on a class of smooth solutions. In \cite{WuYi}, following the ideas of the seminal work of Constantin and Escher \cite{CoEs} on the Camassa-Holm equation, Wu and Yin proved the smooth global well-posedness for initial data with non-negative momentum density.

\begin{thm}[\cite{WuYi}]\label{GWP_smooth}
Let $u_0\in H^s$ for $s\geq 3$, with non-negative  momentum density $y_0$ belonging to $L^{1}(\R)$. Then, equation \eqref{novikov_eq} has a unique global strong solution \[
u\in C(\R,H^s(\R))\cap C^1(\R,H^{s-1}(\R)).
\]
Moreover, denoting by $y(t):=u(t)-u_{xx}(t)$ we have that $E(u)$ and $\Vert y(t)\Vert_{L^{2/3}}$ are two conservation laws. Additionally, we have that $y(t)$ and $u(t)$ are non-negative for all times $t\in\R$ and $\vert u_x(t,\cdot)\vert\leq u(t,\cdot)$ on $\R$.
\end{thm}
Unfortunately, since peakon profiles
do not belong\footnote{Actually, they do not belong to any $W^{1+\frac{1}{p},p}(\R)$ for any $p\in [1,+\infty)$. However, peakon profiles do belong to $W^{1,\infty}(\R)$, where $W^{1,\infty}(\R)$ denotes the space of Lipschitz functions.} to $H^{3/2}(\R)$, they do not enter into this framework either, and hence this theorem is not useful for our purposes. Nevertheless, by following the work of Constantin and Molinet \cite{CoMo}, in the same work Wu and Yin also proved a global well-posedness theorem for a class of functions containing peakons. This result shall be crucial in our analysis. However, we shall need a slightly improved version of this theorem, which we state below. 
\begin{thm}[\cite{WuYi}]\label{theorem_gwp}
Let $u_0\in H^1(\R)$ be a function satisfying $y_0:=(u_0-u_{0,xx})\in\mathcal{M}_b^+(\R)$.
Then, the following properties hold: \begin{enumerate}
\item[1.] \emph{\textbf{Uniqueness and global existence:}} There exists a global weak solution \[
u\in C(\R,H^1(\R))\cap C^1(\R,L^2(\R)),
\]
associated to the initial data $u(0)=u_0$ such that its momentum density \[
y(t,\cdot):=u(t,\cdot)-u_{xx}(t,\cdot)\in C_{ti}(\R,\mathcal{M}_b^+(\R)).
\]
Additionally $I(u)$ and $E(u)$ are conservation laws. Moreover, the solution is unique in the class \[
\{f\in C(\R,H^1(\R))\}\cap\{f-f_{xx}\in L^\infty(\R,\mathcal{M}_b^+)\}.
\]
\item[2.] \emph{\textbf{Continuity with respect to the initial data $H^1(\R)$:}} For any sequence $\{u_{0,n}\}_{n\in\N}$ bounded in $Y_+(\R)$  such that $ u_{0,n}\to u_0 \,\hbox{ in } H^1(\R)$,
the following holds: For any $T>0$, the family of solutions $\{u_{n}\}$ to equation \eqref{nov_eq_2} associated to $\{u_{0,n}\}$ satisfies \begin{align}\label{convergence_h1_ti}
u_n\to u \,\hbox{ in }\, C([-T,T],H^1(\R)) \quad \hbox{and} \quad y_{n}\rightharpoonup y \,\hbox{ in }\, C_{ti}([-T,T],\mathcal{M}).
\end{align}
\end{enumerate}
\end{thm}

\begin{proof}
We refer to \cite{Mo,Mo2}, Propositions $2.2$, for a proof of this theorem in both the Camassa-Holm and the $b$-family case. Notice that the same proof applies to the Novikov equation, provided Theorem \ref{GWP_smooth} and the fact that the first point of the statement was proven in \cite{WuYi}, except for the fact that $y\in C_{ti}(\R,\mathcal{M}_b^+)$, which can be proven in exactly the same fashion as in \cite{Mo}.
\end{proof}

\section{Liouville property for the Novikov equation}\label{sec_MT2}

\subsection{Preliminary properties of almost localized solutions and almost monotonicity lemma}
This subsection aims to state some preliminary properties regarding the decay of almost localized solutions that shall be useful in the sequel. Since the proof of these properties plays no role in the study of Theorems \ref{MT1} and \ref{MT2}, we postpone them to the appendix.
\begin{prop}[Time-uniform exponential decay]\label{MT3}
Let $u\in C(\R,H^1(\R))$ be an $H^1$-almost localized solution to \eqref{nov_eq_2}. Then, there exists a constant $C>0$ only depending on $\Vert u_0\Vert_{H^1}$ and the mapping $\varepsilon\mapsto R_\varepsilon$ (see Definition \ref{def_AL}), and $K\geq1$ such that for all $t\in\R$, all $R>0$ and all $\vert x\vert>R$ we have\begin{align}\label{exp_decay_u}
\vert u\big(t,x+x(t)\big)\vert+\int_{\vert x\vert>R}\left(u^2+u_x^2\right)(t,\cdot+x(t))dx\leq Ce^{-\frac{R}{K}}.
\end{align}
\end{prop}
The previous proposition, whose proof is found in Section \ref{proof_MT_appendix}, is actually a classical consequence of an almost-monotonicity property of the energy (see Lemma \ref{tech_lem_mon_exp}) which, together with the $H^1$-almost localized hypothesis, implies the uniform exponential decay of the solution. To prove this theorem and both Theorem \ref{MT1} and \ref{MT2} let us introduce some useful notation. From now on we shall denote by $\Psi$ the weight function defined by 
\begin{align}\label{psi_def}
\Psi:=\dfrac{2}{\pi}\arctan\left(\exp\big(\tfrac{x}{6}\big)\right),
\end{align}
The idea of introducing this weight function is to measure $u(t,x)$ at the right side of space. Notice that as a direct consequence of the definition we have that $\Psi(x)\to1$ as $x\to+\infty$ and
\begin{align}\label{psi_bound}
0\leq \Psi\leq 1, \qquad \vert\Psi'''\vert\leq\dfrac{1}{10}\Psi' \quad \hbox{and}\quad \forall x\leq 0,\ \, \vert\Psi(x)\vert+\vert \Psi'(x)\vert \lesssim e^{\frac{x}{6}}.
\end{align}
Finally, for any modulation variable $z:\R\to\R$ and any point $x_0\in\R$ we define the modified energy functional 
\begin{align*}
\widehat{\mathrm{I}}_{t_0}(t):=\int _\R\big(u^2(t)+u_x^2(t)\big)\Psi\big(\cdot-x_0-z(t)+z(t_0)\big)dx
\end{align*}
A key point in our analysis is the fact that $\widehat{\mathrm{I}}_{t_0}(t)$ approximates the energy of $u(t)$ at the right of $x(t)=x_0+z(t)-z(t_0)$. Moreover, by using the definition of $\Psi$ in \eqref{psi_def} we deduce that \begin{align}\label{energy_right}
\hbox{for all }\,t_0\in\R, \ \ \widehat{\mathrm{I}}_{t_0}(t_0)>\dfrac{1}{2}\Vert u(t_0,\cdot)\Vert_{H^1(x_0,+\infty)}.
\end{align}
The next technical lemma states the almost monotonicity result of the energy at the right. This lemma shall be crucial in the proofs of Theorems \ref{MT1}-\ref{MT2}, and we shall use it repeatedly.
\begin{lem}[Almost-monotonicity of the energy at the right]\label{tech_lem_mon_exp}
Let $c>0$ and $\delta\in(0,1)$ be two fixed parameters. Assume that $u\in C(\R,H^1(\R))$ with $y\in C_{ti}(\R,\mathcal{M}_b^+)$ is a solution to equation \eqref{novikov_eq} such that there exists $R_0>0$ and a $C^1$ function $x:\R\to\R$ with $\inf_\R\dot{x}(t)\geq c$ satisfying \begin{align}\label{L_infty_bound_outside}
\hbox{for all } \, t\in \R, \ \ \Vert u(t)\Vert_{L^\infty(\vert x-x(t)\vert>R_0)}\leq \dfrac{(1-\delta)c}{\mathbf{b}}, \, \hbox{ where } \, \mathbf{b}:=2^6\max\{1,\Vert u_0\Vert_{H^1}\}.
\end{align}
Then, for $R>R_0$ sufficiently large, $\gamma\in (0,\delta)$ and any $C^1$ function $z:\R\to\R$ satisfying \begin{align}\label{hip_z_dot}
(1-\delta)\dot x(t)\leq\dot z(t)\leq (1-\gamma)\dot x(t), \quad \hbox{for all }\,t\in\R,
\end{align}
the following property holds: Let $t_0\in\R$ be a fixed time. Define the energy functionals \[
\mathrm{I}_{t_0}^{\pm R}(t):=\int_\R\big(u^2(t)+u_x^2(t)\big)\Psi\big(\cdot-z_{t_0}^{\pm R})\big)dx \ \hbox{ where }\ z_{t_0}^{\pm R}(t):=x(t_0)\pm R+z(t)-z(t_0).
\]
Then we have \begin{align}\label{ineq_Itzero_It}
\forall t\leq t_0, \ \, \mathrm{I}_{t_0}^R(t_0)-\mathrm{I}_{t_0}^R(t)\leq Ce^{-R/6} \  \hbox{ and } \ \ \forall t\geq t_0, \ \, \mathrm{I}_{t_0}^{-R}(t)-\mathrm{I}_{t_0}^{-R}(t_0)\leq Ce^{-R/6},
\end{align}
for some constant $C>0$ only depending on $\delta$, $\gamma$, $c$, $R_0$ and $E(u)$.
\end{lem}

\begin{rem}
Notice that we are not assuming that $u(t)$ is an $H^1$-almost localized solution. This shall important to study limit objects in Section \ref{sec_MT1}, where hypothesis \eqref{L_infty_bound_outside}-\eqref{hip_z_dot} shall be guaranteed by a modulation argument.
\end{rem}

\begin{proof}
See the appendix, Section \ref{tech_lem_appendix}.
\end{proof}

\subsection{Comments on the method of proof of Theorem \ref{MT2}}
Before going further, for the sake of clarity, let us sketch the ideas of the proof of Theorem \ref{MT2}. We shall proceed as follows: First, we start by studying properties of solutions with compactly supported momentum density. In particular, we shall prove that for this class of solutions there exists a Lyapunov functional, which is related to the last point on the support of the momentum density. Then, in the next section, we shall prove that every $H^1$-almost localized solution of equation \eqref{nov_eq_2} has compactly supported momentum density, and hence all the properties proved in the previous section hold. This shall be a consequence of the finite speed of propagation of the momentum density and the time-uniform exponential decay of $H^1$-almost localized solutions. Then, by using the Lyapunov functional we shall prove that $u(t)$ evaluated at the integral line associated to the last point of the support of the momentum density is constant in time. Finally, we show that the latter fact forces $u(t)$ to be a peakon.

\subsection{A Lyapunov functional for solutions with compactly supported momentum density}\label{jump_section}

In this section we shall assume that we are working with a solution of equation \eqref{nov_eq_2} such that the support of its momentum density is bounded from above. In the next sections we shall prove that almost localized solutions enjoy this property.

\medskip

Before going further, we need to introduce the flow $q$ associated with $u^2$, which is  defined by
\begin{align}\label{ODE_q_nov}
\begin{cases} 
q_{t}(t,x)=u^2\big(t,q(t,x)\big),
\\ q(0,x)=x.
\end{cases}
\end{align}
From \cite{WuYi2} we know that the solutions associated to this ODE satisfy, for every $t\in\R$,
\begin{align}\label{eq_flow}
y\left(t,q(t,x)\right)q_x(t,x)^{\frac{3}{2}}=y_0(x)
\end{align}
Moreover, by differentiating \eqref{ODE_q_nov} with respect to $x\in\R$ we also obtain
\begin{align}\label{derivative_flow_line}
q_x(t,x)=\exp\left(2\int_0^t u\big(s,q(s,x)\big)u_x\big(s,q(s,x)\big)ds\right).
\end{align}

Now we intend to study what consequences the existence of this last point on the support of $y(t)$ has. In this regard, we shall need the following definition
\[
x_+(t):=\inf\left\{x\in\R: \ \supp y(t)\subseteq(-\infty,x(t)+x]\right\}.
\]
We emphasize that during this section we are assuming that $x_+(\cdot)$ is well-defined. The following lemma show us that under these assumptions the map $t\mapsto x(t)+x_+(t)$ is actually an integral line of $u^2(t)$.
\begin{lem}\label{integral_line_nov}
Suppose that $u\in C(\R,H^1(\R))$ is an $H^1$-almost localized solution of \eqref{nov_eq_2} with $\inf_\R\dot{x}\geq c>0$. Moreover, assume that there exists $r\in\R$ such that for all $t\in\R$ it holds \begin{align}\label{supp_y_prop_nov}
\supp y\big(t,\cdot+x(t)\big)\subset(-\infty,r].
\end{align}
Then, for all $t\in\R$, we have \begin{align}\label{supp_and_integral}
x(t)+x_+(t)=q\big(t,x(0)+x_+(0)\big),
\end{align}
where $q(\cdot,\cdot)$ is defined by \begin{align}\label{def_q_nov}
\begin{cases}
q_t(t,x)=u^2\big(t,q(t,x)\big), & (t,x)\in\R^2,
\\ q(0,x)=x, & x\in\R.
\end{cases}
\end{align}
Additionally, for all $t\in\R$ and $z\geq x_+(t)$ we have \begin{align}\label{equality_supp}
u\big(t,x(t)+z\big)=-u_x\big(t,x(t)+z\big).
\end{align}
\end{lem}

\begin{proof}
See the appendix, Section \ref{appendix_integral_line}.
\end{proof}

In the sequel we shall need the following definitions associated to the operator $(1-\partial_x^2)^{-1}$. From now on we denote by $p_+$ and $p_-$ the following operators
\[
p_+*f(x):=\dfrac{e^{-x}}{2}\int_{-\infty}^x e^zf(z)dz \quad \hbox{and}\quad p_-*f(x):=\dfrac{e^{x}}{2}\int_{x}^\infty e^{-z}f(z)dz.
\]
Note that $p=p_++p_-$. The following crucial lemma give us the existence of a Lyapunov functional for solutions with compactly supported momentum density.
\begin{lem}
Under the hypothesis of Lemma \ref{integral_line_nov}, the map $t\mapsto u\big(t,x(t)+x_+(t)\big)$ defines a bounded increasing function on $\R$.
\end{lem}

\begin{proof}
The proof follows from some direct computation together with a regularization argument and  \eqref{equality_supp}. We point out that the computation of the time derivative along characteristics has already been made in \cite{ChGuLiQu}.

\medskip

Let $\varepsilon>0$ small enough. We set the point $x_\varepsilon=x(0)+x_+(0)+\varepsilon$. Now, we define the integral line associated to $x_\varepsilon$, that is, $x_\varepsilon(t)=q(t,x_\varepsilon)$, where $q(\cdot,\cdot)$ is defined in \eqref{def_q_nov}.

\medskip

Now, notice that $\Vert u_x(t)\Vert_{L^\infty}\leq C$, and hence $u(t)$ is Lipschitz continuous with respect to the space variable. Therefore, by using Cauchy-Lipschitz's Theorem for ODEs we deduce that \[
x_\varepsilon(\cdot)\to x(\cdot)+x_+(\cdot) \ \hbox{ in } \ C(\R) \ \hbox{ as } \ \varepsilon\to 0.
\]
Moreover, since $u(t)$ is continuous the latter convergence result implies that
\begin{align}\label{conv_x_varepsilon}
u(\cdot,x_\varepsilon(\cdot))\to u(\cdot,x(\cdot)+x_+(\cdot)) \ \hbox{ in } \ C(\R) \ \hbox{ as } \ \varepsilon\to 0.
\end{align}
Notice that by using \eqref{equality_supp} together with the previous convergence result we also obtain that \[
u_x(\cdot,x_\varepsilon(\cdot))\to -u(\cdot,x(\cdot)+x_+(\cdot)) \ \hbox{ in } \ C(\R) \ \hbox{ as } \ \varepsilon\to 0.
\]
On the other hand, since $\supp y(t)\subseteq (-\infty,x(t)+x_+(t)]$, we deduce that $u(t)$ is $H^3$ in a neighborhood of $x_\varepsilon(t)$. Therefore, by using the equation and recalling that $(p*\cdot):L^2\to H^2$, we conclude that $u_t(t)$ is differentiable with respect to $x$ in a neighborhood of $x_\varepsilon(t)$. 

\medskip

Now we intend to compute the time-derivative of $u$ along the integral line $x_\varepsilon(t)$. For the sake of simplicity we shall actually compute the time-derivative of the map $t\mapsto u_x(t,x_\varepsilon(t))$, which turns out to be easier. In fact, using Lemma \ref{integral_line_nov} we get 
\begin{align*}
\dfrac{d}{dt}u_x\big(t,x_\varepsilon(t)\big)&=u_{tx}+u^2u_{xx}=-\dfrac{1}{2}uu_x^2+u^3-p*\Big(\dfrac{3}{2}uu_x^2+u^3\Big)-\dfrac{1}{2}p_x*u_x^3
\end{align*}
Thus, after integration by parts and by using the operators $p_\pm$ we obtain \begin{align}\label{dt_charact_u}
\dfrac{d}{dt}u_x(t,x_\varepsilon(t))=\dfrac{1}{2}u(u^2-u_x^2)-\dfrac{1}{2}(p_+*(u-u_x)^3+p_-*(u+u_x)^3),
\end{align}
where all the right-hand side is evaluated at $x=x_\varepsilon(t)$. Hence, by using \eqref{equality_supp} and due to the fact that the kernels of $p_+$ and $p_-$ are both positive and that $\vert v_x\vert\leq v$ for any $v\in Y_+$, we deduce \begin{align*}
\dfrac{d}{dt}u_x(t,x_\varepsilon(t))\leq0, \quad \hbox{and therefore}\quad \dfrac{d}{dt}u(t,x_\varepsilon(t))\geq 0.
\end{align*}
Thus, $u(t,x_\varepsilon(t))$ is increasing, and hence, by using the convergence result \eqref{conv_x_varepsilon} we conclude that $u(t,x(t)+x_+(t))$ is increasing, what finish the proof of the lemma.
\end{proof}
As a corollary of the previous analysis we obtain the following key property. 
\begin{cor}\label{prop_limit}
Both maps $x(t)+x_+(t)$ and $\dot{x}(t)+\dot{x}_+(t)$ define non-decreasing functions. Moreover, these are $C^1$ and $C^0$ functions respectively and there exists $c_\pm\geq0$ such that \[
\lim_{t\to\pm\infty}\dot{x}(t)+\dot{x}_+(t)=\lim_{t\to\pm\infty}u^2(t,x(t)+x_+(t))\to c_\pm.
\]
\end{cor}
\begin{proof}
This is just a consequence of Lemma \ref{integral_line_nov} and the fact that $u(t,x(t)+x_+(t))$ is monotone on $\R$ and bounded, and hence we immediately conclude the existence of both limits at $\pm\infty$. 
\end{proof}

\subsection{Almost localized solutions have momentum density with compact support}

The following property ensures that the momentum density associated with an $H^1$-almost localized solution is compactly supported. This is the key fact of our proof. Notice that once we prove this property, all the results in Section \ref{jump_section} hold for $y(t)=(u-u_{xx})(t)$.

\begin{prop}\label{prop_compact_support}
Suppose that $u\in C(\R,Y_+)$ is an $H^1$-almost localized solution to equation \eqref{nov_eq_2}. Then, there exists $r\in\R$ such that for all $t\in\R$ it holds \begin{align}\label{compact_support}
\supp y(t,\cdot+x(t))\subset(-\infty,r].
\end{align}
\end{prop}

\begin{proof}
First of all notice that it is enought to prove the result for $t=0$. Moreover, notice that due to the fact that $y\in \mathcal{M}_b^+$, it is enough to prove the following property: Let $\phi\in C^\infty(\R)$ any function satisfying \[
\phi(x)\equiv 0 \,\hbox{ for }\, x\in\R_-, \quad \phi(x)\equiv1 \,\hbox{ for }\, x\in [1,\infty) \quad \hbox{and}\quad \phi'(x)\geq 0\ \forall x\in\R.
\]
Then, there exists $r^\star\in\R$ sufficiently large such that the following equality holds \begin{align}\label{contr_prop_supp}
\langle y(0),\phi(\cdot-x(0)-r^\star)\rangle=0.
\end{align}
Now, in order to prove \eqref{contr_prop_supp} we start by approximating $u_0$ by a sequence of smooth functions \[
u_{0,n}:=\rho_n*u_0\in H^\infty(\R)\cap Y_+(\R) \ \, \hbox{ and }\ \, y_{0,n}\rightharpoonup y_0 \ \hbox{ in }\ \mathcal{M},
\]
so that \eqref{convergence_h1_ti} holds for any $T>0$. We emphasize that the latter weak convergence is in the sense of Definition \ref{def_weakly_conv}. Notice that by Theorem \ref{GWP_smooth} we obtain that the solution $u_n(t)$ associated to $u_{0,n}$ belongs to $C(\R,H^\infty(\R))$ and its momentum density $y_n\in C_{ti}(\R,L^1(\R))$. 

\medskip 

On the other hand, notice also that for all $n\in\N$, the solution $u_n(t)$ is also $H^1$-almost localized with the same localizing function $x\mapsto x(t)$ and a radius $R_\varepsilon^n$ that converges to $R_\varepsilon$ as $n\to+\infty$. Moreover, since  the mollifier family $\rho_n$ have compact supports, by adding an universal constant to the one in front of the exponential in \eqref{exp_decay_u} 
we conclude that all the sequence $u_n(t)$ have the same time-uniform exponential decay, i.e., \[
u_n(t,\cdot+x(t))\leq C_*\exp(-\vert x\vert/K), \quad \hbox{for all }\ n\in\N,
\]
for some constant $C_*>C$ and some $K\geq 1$.

\medskip

On the other hand, notice that for any fixed $T>0$, there exists $n_0\in\N$ such that for all $n\geq n_0$ the following inequalities holds \begin{align}\label{approx_linfty}
\Vert u_{n}-u\Vert_{L_T^\infty H^1_x}<\dfrac{1}{10}\min\{\sqrt{c},\Vert u_0\Vert_{H^1}\}.
\end{align}
Moreover, by the $H^1$-almost localized hypothesis we deduce that there existe $r> 1$ sufficiently large such that \begin{align}\label{mthm_Linfty}
\Vert u(t)\Vert_{H^1(\vert x-x(t)\vert>r-1)}\leq \dfrac{1}{10}\min\left\{ \dfrac{c}{2^6},\sqrt{c},\Vert u_0\Vert_{H^1}\right\}, \quad \hbox{for all }\, t\in\R.
\end{align}
Notice that due to Sobolev's embedding, inequality \eqref{approx_linfty} implies that for all $n\geq n_0$ we have \begin{align}\label{sob_un}
u_n(t,x+x(t))\leq \dfrac{1}{5}\min\{\sqrt{c},\Vert u_0\Vert_{H^1}\} \quad \hbox{for all} \quad (\vert x\vert,t)\in [r-1,+\infty)\times [-T,T].
\end{align}
Finally, we need to introduce the flow $q_n$ associated to our approximate solution $u^2_n$,
\begin{align}\label{mthm_ODE_q}
\begin{cases} 
q_{t,n}(t,x)=u_n\big(t,q_n(t,x)\big)^2,
\\ q_n(0,x)=x.
\end{cases}
\end{align}
We recall that (see \cite{WuYi2}) the solutions associated to this ODE satisfy, for every $t\in\R$,
\begin{align}\label{mthm_eq_flow_main}
y_n\big(t,q_n(t,x)\big)q_{x,n}(t,x)^{\frac{3}{2}}=y_n(0,x)
\end{align}
Moreover, by differentiating \eqref{mthm_ODE_q} with respect to $x\in\R$ we also obtain
\begin{align}\label{mthm_derivative_flow_line}
q_{x,n}(t,x)=\exp\left(2\int_{0}^{t} u_n\big(s,q_n(s,x)\big)u_{x,n}\big(s,q_n(s,x)\big)ds\right).
\end{align}
Now we claim that due to the $H^1$-almost localization of $u(t)$ we have 
\begin{align}\label{claim_support_nov}
q_n\big(t,x(0)+r\big)-x(t)\geq r+\dfrac{c\vert t\vert}{2} \quad \hbox{and} \quad \dfrac{1}{C_0}\leq q_{x,n}\big(t,x(0)+r+x\big)\leq C_0,
\end{align}
for some $C_0>0$. For the sake of simplicity we shall show this fact at the end of the proof. Thus, assuming the previous inequalities and by using \eqref{mthm_eq_flow_main} we deduce
\begin{align*}
\int_{x(0)+r-1}^\infty y_n(0,x)dx&=\int_{x(0)+r-1}^{+\infty} y_n\big(t,q_n(t,x)\big)q_{x,n}^{3/2}(t,x)dx
\\ & \leq C_0^{1/2}\int_{x(0)+r-1}^{+\infty}y_n\big(t,q_n(t,x)\big)q_{x,n}(t,x)dx
\end{align*}
Now, by using \eqref{claim_support_nov} together with the uniform exponential decay of both $u_n(t,x+x(t))$ and $u_{n,x}(t,x+x(t))$ (see Proposition \ref{MT3}) we obtain 
\begin{align*}
\int_{x(0)+r-1}^{+\infty}y_n\big(t,q_n(t,x)\big)q_{x,n}(t,x)dx&\leq\int_{r-1+\frac{c}{2}\vert t\vert}^{+\infty}y_n(t,x+x(t))dx 
\\ & \lesssim e^{(r-1+c\vert t\vert/2)/K}+\int_{r-1+\frac{c}{2}\vert t\vert }^{+\infty} e^{-\frac{\vert x\vert}{K}}dx \xrightarrow{t\to-\infty}0.
\end{align*}
Therefore, due to the sequential weak-lower semicontinuity given in Lemma \ref{lem_measures_2} and the positivity of $y_0$ and $y_n$ for all $n\in\N$, we have that \begin{align*}
\langle y(0),\phi(\cdot-x(0)-r)\rangle&\leq \liminf_{n\to+\infty}\int_{x(0)+r-1}^{+\infty}y_n(0,x)dx=0,
\end{align*}
and hence, up to the proof of both inequalities in \eqref{claim_support_nov}, we conclude the proof of the proposition.

\medskip

\textit{Proof of \eqref{claim_support_nov}:} The proof is straightforward in some sense and only requires to integrate. In fact, it is enough to notice that due to  \eqref{mthm_Linfty} and Sobolev's embedding we have \[
u_n(t,q_n(t,x(0)+r))\leq \dfrac{\sqrt{c}}{4}.
\]
Hence, by plugging the latter inequality into \eqref{mthm_ODE_q} and using the fact that $\inf_{\R}\dot{x}(t)\geq c$ we infer that for all $t<0$ we have \[
\dfrac{d}{dt}q_n(t,x(0)+r)\leq\dfrac{c}{16} \quad \hbox{ which implies } \quad q_n(t,x(0)+r)-x(t)\geq r+\dfrac{c}{2}\vert t\vert,
\]
what finish the proof of the first inequality in \eqref{claim_support_nov}. 

\medskip

Finally, let us prove the boundedness of $q_x(t)$. Recall that due to the $H^1$-almost localization hypothesis and by using Proposition \ref{MT3} we have, in particular, that $u_n(t)$ has time-uniform exponential decay. Thus, by using the exponential decay of $u_n(t)$, the almost monotonicity of the energy, Sobolev's embedding and inequality \eqref{claim_support_nov} we deduce that for any $s\in \R$ it holds
\begin{align*}
u^2_n\big(s,q(s,x(0)+r)\big)&\leq\sup_{x\geq x(s)+r+\frac{1}{2}c\vert s\vert}u_n^2(s,x)\lesssim \exp\left(-\frac{2r+c\vert s\vert}{K} \right).
\end{align*}
Hence, due to the latter inequality and the fact that $\vert v_x\vert\leq v$ for any $v\in Y_+$ we obtain 
\[
\int_0^{+\infty} u_n\big(s,q_n(s,x(0)+r)\big)u_{x,n}\big(s,q_n(s,x(0)+r)\big)ds\leq C.
\]
Therefore, by plugging the latter inequality into formula \eqref{mthm_derivative_flow_line}  we deduce the existence of a constant $C_0>0$ such that $\tfrac{1}{C}\leq q_{x,n}\big(t,x(0)+r\big)\leq C$ for all $t\in\R$, what ends the proof. 
\end{proof}

\subsection{Proof of Theorem \ref{MT2}}

In this section we assume that we are under the hypothesis of Theorem \ref{MT2}. 

\medskip

Motivated by the study made in Section \ref{jump_section}, we define $x_+$ the corresponding quantity which give us the position of the last point on the support of $y(t)$, that is,
\[
x_+(t):=\inf\left\{x\in\R: \ \supp y(t)\subseteq(-\infty,x(t)+x]\right\}.
\]
Note that Proposition \ref{prop_compact_support} ensures that the map $t\mapsto x_+(t)$ is well-defined and bounded from above.

\medskip

Before getting into the details, let us start by explaining the idea of the proof: We shall proceed in two steps: First, we intend to prove that  $u(t,x(t)+x_+(t))$ does not depend on time, i.e. it is constant. Then, we shall prove that this property forces $u$ to achieve an equality only achievable by peakons.

\medskip

The following technical (but straightforward) lemma shall be crucial in the proof of Proposition \ref{prop_312} below. We postpone their proofs for the Appendix.
\begin{lem}\label{tech_ineq_lem}
Let $v\in Y_+$. Then, the following inequality holds \begin{align}\label{point_conv}
p*\left(3 vv_x^2+5 v^3\right)\geq 2v^3(x), \quad \forall x\in\R.
\end{align}
Moreover, equality holds in \eqref{point_conv} for some $x_0\in\R$ if and only if $v(\cdot)$ is a peakon, that is, there exists $c\in\R_+$ such that $v(x)=\sqrt{c}e^{-\vert x-x_0\vert}$.
\end{lem}

\begin{proof}
See the appendix, Section \ref{tech_ineq_lem_appendix}.
\end{proof}


\begin{prop}\label{prop_312}
Let $u\in C(\R,H^1)$ be an $H^1$-almost localized solution to equation \eqref{nov_eq_2}. Then, $u(t)$ must to be a peakon.
\end{prop}

\begin{proof}
First of all, we recall that by space-time reflection invariance we know that if $u(t,x)$ is a solution to \eqref{nov_eq_2}, then so is $v(t,x):=u(-t,-x)$. Moreover, notice that from the definition of $v(t)$ it is direct to check that \[
v\in C(\R,H^1(\R)) \quad \hbox{and}\quad v-v_{xx}\in C_{ti}(\R,\mathcal{M}_b^+),
\]
and hence $v$ is a $H^1$-almost localized solution of \eqref{nov_eq_2} localized at $\widetilde{x}(\cdot):=-x(-\cdot)$ and the same mapping $\varepsilon\mapsto R_\varepsilon$. On the other hand, by applying Lemma \ref{integral_line_nov} and Corollary \ref{prop_limit} we deduce the existence of a positive constant $\widetilde{r}\in\R$ and a $C^1$ function $\widetilde{x}_+:\R\to(-\infty,\widetilde{r}]$, such that \[
\lim_{t\to\pm\infty} v\big(t,\widetilde{x}(t)+\widetilde{x}_+(t)\big)=\sqrt{\widetilde{c}_\pm} \quad \hbox{or equivalently} \quad \lim_{t\to\mp\infty} u\big(t,x(t)-\widetilde{x}_+(-t)\big)=\sqrt{\widetilde{c}_\pm}.
\]
\textbf{Step 1:} We claim that this implies $c_+=c_-=\widetilde{c}_+=\widetilde{c}_-$. Indeed, first of all notice that from the monotonicity of $t\mapsto u(t,x(t)+x_+(t))$ we have $\widetilde{c}_-\leq\widetilde{c}_+$ and $c_-\leq c_+$. Now, let us prove by contradiction that $c_+\leq \widetilde{c}_-$. In fact, if this were not true, then there would exists $t_0\in\R$ such that for all $t\geq t_0$ we would have
\begin{align}\label{contr_proof_mt2}
u\big(t,x (t)-\widetilde{x}_+(-t)\big)<u \big(t,x (t)+x_+(t)\big)-\varepsilon,
\end{align}
for some $\varepsilon>0$. On the other hand, notice that by Lemma \ref{integral_line_nov} it holds
\[
x(t)+x_+(t)=q(t-t_0,x(t_0)+x_+(t_0)),\] 
and 
\[
x(t)-\widetilde{x}_+(-t)=q(t-t_0,x(t_0)-\widetilde{x}_+(-t_0)).
\]
Thus, by using \eqref{def_q_nov} and \eqref{contr_proof_mt2} we obtain \begin{align*}
x_+(t)+\widetilde{x}_+(-t)&=x_+(t_0)+\widetilde{x}_+(-t_0)+\int_{t_0}^t q_t(\tau-t_0,x(t_0)+x_+(t_0))d\tau
\\ & \quad -\int_{t_0}^t q_t(\tau-t_0,x(t_0)-\widetilde{x}_+(-t_0))d\tau
\\ & \geq \varepsilon(t-t_0)+x_+(t_0)+\widetilde{x}_+(-t_0).
\end{align*}
Since the right-hand side goes to $+\infty$ as $t\to+\infty$, this contradicts the fact that, by Proposition \ref{compact_support}, both $x_+(t)$ and $\widetilde{x}_+(t)$ are bounded from above. Notice that in the same fashion we also obtain that $\widetilde{c}_+\leq c_-$, what ends the proof of the claim. Therefore, we conclude that \[
u\big(t,x(t)+x_+(t)\big)\equiv \sqrt{c_+} \quad \forall t\in\R.
\]
\textbf{Step 2:} Now, we claim that this forces $u(t)$ to be a peakon. We proceed by contradiction, that is, let us assume that $u(t)$ is not a peakon, and hence by Lemma \ref{tech_ineq_lem}, inequality \eqref{point_conv} is strictly satisfied. We claim that this forces the following strict inequality to hold at $x=x(t)+x_+(t)$
\begin{align}\label{ineq_claim_mt}
-\dfrac{1}{2}u u_x^2+u^3-p*\left(\dfrac{3}{2}u u_x^2+u^3\right)-\dfrac{1}{2}p_x*u_x^3<-2p*u^3, \quad \forall t\in\R. 
\end{align}
In fact, by using Lemma \ref{tech_ineq_lem} we know that for all $t\in\R$ it holds \[
u^3\big(t,x(t)+x_+(t)\big)-p*\left(\dfrac{3}{2}u u_x^2+\dfrac{5}{2}u^3\right)\big(t,x(t)+x_+(t)\big)<0.
\]
On the other hand, since $\vert v_x\vert\leq v$ for any $v\in Y_+$ we have \[
-\dfrac{1}{2}u u_x^2+\dfrac{1}{2}p*u^3+\dfrac{1}{2}p_x*u_x^3\leq 0.
\]
Hence, gathering the last two inequalities we conclude the claim. 

\medskip

\textbf{Step 3:} Now we intend to use the latter two steps to conclude that $u(t)$ must to be a peakon. We recall that we are still in the contradiction argument of Step $2$, so inequality \eqref{ineq_claim_mt} holds.

\medskip

In fact, by using \eqref{equality_supp} together with formula \eqref{dt_charact_u} and the previous claim we obtain \[
\dfrac{d}{dt}u_x\big(t,x (t)+x_+(t)\big)< -2p*u^3 \quad \hbox{or equivalently} \quad  \dfrac{d}{dt}u \big(t,x (t)+x_+(t)\big)> 2p*u^3.
\]
On the other hand, recall that if $u(t)\in Y_+$, then $u(t)\geq 0$ on $\R$. Therefore, the latter inequality together with the non-negativity of $u(t)$ and $p(x)$ implies that  \[
\dfrac{d}{dt}u\big(t,x(t)+x_+(t)\big)>0, \ \hbox{ for all } \ t\in\R,
\]
but this contradicts the fact that $u\big(t,x(t)+x_+(t)\big)\equiv\sqrt{c_+}$. Therefore, $u(t)$ must to be a peakon. The proof is complete. 
\end{proof}

\section{Proof of Theorem \ref{MT1}}\label{sec_MT1}

\subsection{Modulation around peakons} In the sequel we shall closely follow the approach made by Molinet in \cite{Mo} (see also \cite{MaMe1,MaMe2,MaMeTs} for previous results using this approach for different equations).  From now on we assume we are in the context of Theorem \ref{MT1}, that is, from now on let us assume that there exists $c>0$ and $u_0\in Y_+$ such that \begin{align}\label{pequeno}
\Vert u_0-\sqrt{c}\varphi\Vert_{H^1}\leq \sqrt{c}\varepsilon^8, \quad \hbox{for some} \quad 0<\varepsilon<c.
\end{align}
Then, according to the orbital stability result for peakon soltuions (see \cite{LiLiQu}), there exists a function $\xi:\R\to\R$ such that the global solution $u(t)$ to equation \eqref{nov_eq_2} associated to $u_0$ satisfies \begin{align}\label{conclusion_orb_stab}
\sup_{t\in\R}\Vert u(t)-\sqrt{c}\varphi(\cdot-\xi(t))\Vert_{H^1}\lesssim\mathbf{c}\varepsilon^2, \quad \mathbf{c}:=\min\{\sqrt{c},\sqrt[8]{c}\},
\end{align}
where $\xi(t)\in\R$ corresponds to any maximum point of $u(t,\cdot)$ and the implicit constant only depends on\footnote{In \cite{LiLiQu} the implicit constant appearing in \eqref{conclusion_orb_stab} does depends on $\Vert u_0\Vert_{H^3}$. Nevertheless, it is easy to check that it can actually be sharpened to depend only on $\Vert u_{0,x}\Vert_{L^\infty}$. Here, since $u_0\in Y_+$ we have $\Vert u_{0,x}\Vert_{L^\infty}\leq \Vert u_0\Vert_{L^\infty}\leq \Vert u_0\Vert_{H^1}$. The interested reader can consult to \cite{Pa} for a simplification of this proof without the sign assumption of the momentum density.} $\Vert u_0\Vert_{H^1}$. Before going further we shall need a modulation lemma for solutions close to a peakon.
\begin{lem}\label{modulational_lemma}
There exists $\varepsilon_0>0$ small enough, $C>1$, $\sigma>0$ and $n_0\in\N$ such that if a solution $u\in C_{ti}(\R,Y_+(\R))$ to equation \eqref{nov_eq_2} satisfies \begin{align}\label{hip_lem_mod}
\sup_{t\in\R}\Vert u(t)-\sqrt{c}\varphi(\cdot-z(t))\Vert_{H^1}\leq \sqrt{c}\varepsilon_0,
\end{align}
for some function $z:\R\to\R$ then the following properties hold: There exists a unique $C^1$ function $x:\R\to\R$ such that \begin{align}\label{mod_rho_close}
\sup_{t\in\R}\vert x(t)-z(t)\vert<\sigma \quad \hbox{and}\quad \int_\R u(t)\big(\rho_{n_0}*\varphi'\big)(\cdot-x(t))=0, \quad \forall t\in\R,
\end{align} 
where $\rho_n$ is defined in \eqref{def_rho} and $n_0\in\N$ satisfies: \begin{align}\label{orth_cond_def}
\hbox{For all }\,-\tfrac{1}{2}\leq y\leq \tfrac{1}{2}, \quad \int_\R\varphi(\cdot-y)(\rho_{n_0}*\varphi')=0 \quad \iff \quad y=0.
\end{align}
Moreover, the function satisfies \begin{align}\label{uniform_bound_param_mod}
\sup_{t\in\R}\vert \dot{x}(t)-c\vert<\dfrac{c}{8}
\end{align}
Additionally, let $0<\varepsilon<c\varepsilon_0$, then the following property holds: \begin{align}\label{uniform_mod}
\hbox{if }\ \sup_{t\in\R}\Vert u(t)-\sqrt{c}\varphi(\cdot-z(t))\Vert_{H^1}<\tfrac{\varepsilon^2}{c^{3/2}}\quad \hbox{then} \quad \sup_{t\in\R}\Vert u(t)-\sqrt{c}\varphi(\cdot-x(t))\Vert_{H^1}\leq C\varepsilon.
\end{align}
\end{lem}
\begin{proof}
The existence and regularity of $x(t)$ is a standard application of the Implicit Function Theorem. We postpone this proof for the appendix (see Section \ref{apendix_mod}).
\end{proof}
At this point, let us consider $\beta\in(0,c)$ fixed. We define $\mathbf{d}:=\max\{\mathbf{c}^{3/2},\mathbf{c}^{-3/2}\}$ and \begin{align}\label{varepsilon_star}
\varepsilon_*:=\tfrac{1}{2C}\min\left\{\tfrac{\beta}{2^8},\sqrt{c}\varepsilon_0,\tfrac{1}{2^{8}\mathbf{d}}\right\},
\end{align}
where $C>0$ is the constant involved in \eqref{uniform_mod}. Due to the orbital stability result, we infer that if $u_0\in Y_+$ satisfies \eqref{pequeno} with this $\varepsilon_*$, then  \eqref{conclusion_orb_stab} ensure us that \eqref{hip_lem_mod} is satisfied and hence \eqref{uniform_bound_param_mod} holds. Thus, we obtain that $\dot{x}(t)\geq \tfrac{3}{4}c$ and hence $u(t)$ fulfills the hypothesis of Lemma \ref{tech_lem_mon_exp} for any $\delta\in(0,1)$ satisfying $1-\delta\geq \tfrac{\beta}{4c}$. Notice that in particular we can choose $\delta=\tfrac{1}{3}$. Moreover, notice that by defining $\varepsilon^\star$ as \[
\varepsilon^\star=\tfrac{1}{C^8}\min\left\{\dfrac{1}{2^{10}\mathbf{d}},\dfrac{\varepsilon_0}{6\mathbf{d}}\right\}^8,
\]
we infer that hypothesis \eqref{smallness_hip} implies that inequality \eqref{pequeno} is satisfied with $\varepsilon_*$ as in \eqref{varepsilon_star}.

\subsection{Comments on the proof of Theorem \ref{MT1}}
Before going further let us sketch the ideas of the proof of Theorem \ref{MT1}. We shall proceed as follows: First, we start by studying limiting objects associated to $u(t,\cdot+x(t))$ where $u(t)$ corresponds to our original solution. We shall prove that these limit objects enjoy better properties than the solution itself. In particular, we shall prove that solutions associated to this class of limit functions corresponds to $H^1$-almost localized solutions, and hence they are peakons. 

\medskip

The main difficulty here in comparison with the Camassa-Holm and the Degasperis-Procesi equations is that our sequence $u(t_n,\cdot+x(t_n))$ associated to our original solution shall not be bounded in $Y_+$ due to the non-conservation of the momentum. In consequence, we shall not be able to use a general continuity result of the flow-map with respect to the weak topology for bounded sequences in $Y_+$, as it was the case for these last equations. Instead, we shall take advantage of the almost monotonicity result to prove that for a solution $u$ staying close enough to a peakon, the limit objects associated to  $t\mapsto u(t_n+t,\cdot+x(t_n+t))$ is uniformly exponentially decaying, and thus has a finite momentum. This shall be enough to ensure the weak continuity of the flow-map with respect to our sequence. 

\medskip

Once we know that the limit object corresponds to a peakon, we shall be able to slightly improve our previous strong convergence result in $H^{1^-}_{loc}$ to strong convergence in $H^1_{loc}$. Finally, with this latter property, together with the modulation lemma and the weak convergence in the whole space $H^1$, we shall conclude the proof of Theorem \ref{MT1}.

\subsection{Study of limit solutions}
In the rest of this paper we shall need to explicitly study the behavior of the solution $u(t)$ on both, the left and right part of the space. Let us start recalling the definition of the weight function $\Psi$ given in \eqref{psi_def}: \begin{align}\label{psi_def_2}
\Psi(x)=\dfrac{2}{\pi}\arctan\left(\exp(\tfrac{x}{6})\right), \quad \hbox{so that}\quad \Psi(x)\to 1 \ \hbox{ as }\ x\to+\infty.
\end{align}
Before going further we shall need to introduce some additional notation. For $v\in Y$ and $R>0$ we define the functionals $\mathcal{J}_l^R$ and $\mathcal{J}_r^R$ given by \begin{align*}
\mathcal{J}_r^R=\left\langle v^2+v_x^2,\Psi(\cdot-R)\right\rangle \ \hbox{ and }\ \, \mathcal{J}_l^R=\left\langle v^2+v_x^2,1-\Psi(\cdot+R)\right\rangle.
\end{align*}
Now we fix $t_0\in\R$ and let $\gamma=\tfrac{1}{3}$. Considering $z(t)=\tfrac{2}{3}x(t)$ we have that $z(t)$ satisfies condition \eqref{hip_z_dot} and hence we obtain \[
\mathcal{J}_r^R\big(u(t,\cdot+x(t))\big)\geq \mathrm{I}_{t_0}^R(t), \quad \forall t\leq t_0,
\]
where $\mathrm{I}_{t_0}^R(t)$ is the functional defined in Lemma \ref{tech_lem_mon_exp}. Moreover, notice that in particular we have $\mathcal{J}_r^R\big(u(t_0,\cdot+x(t_0))\big)=\mathrm{I}_{t_0}^R(t_0)$. Thus, by using \eqref{ineq_Itzero_It} we deduce \begin{align}\label{ineq_J_r}
\mathcal{J}_r^R\big(u(t_0,\cdot+x(t_0))\big)\leq \mathcal{J}_{r}^R\big(u(t,\cdot+x(t))\big)+C e^{-\frac{R}{6}}, \quad \forall t\leq t_0,
\end{align}
where $C>0$ is the constant appearing in \eqref{ineq_Itzero_It}. On the other hand,  for the sake of notation we also introduce the functional $\widetilde{\mathrm{I}}_{t_0}^R(t)$ given by
\[
\widetilde{\mathrm{I}}_{t_0}^R(t):=\left\langle u^2+u_x^2,1-\Psi\big(\cdot-\tfrac{1}{3}x(t_0)+R-\tfrac{2}{3}x(t)\big)\right\rangle=E(u)-\mathrm{I}_{t_0}^{-R}(t).
\]
Notice that due to the energy conservation together with inequality \eqref{ineq_Itzero_It} it holds \begin{align}\label{a_m_left_energy}
\widetilde{\mathrm{I}}_{t_0}^R(t)\geq\widetilde{\mathrm{I}}_{t_0}^R(t_0)-Ce^{-R/6}.
\end{align}
Therefore, for all $t\geq t_0$ we have \begin{align}\label{ineq_J_l}
\mathcal{J}_l^R\big(u(t,\cdot+x(t))\big)\geq \mathcal{J}_l^R\big(u(t_0,\cdot+x(t_0))\big)-Ce^{-\frac{R}{6}}.
\end{align}
With these definitions at hand, we can get into the proof of the main theorem of this paper. As we already discussed, the proof of theorem \ref{MT1} consists of studying limiting objects which enjoy better properties than the solution itself. The following property ensures that the $\omega$-limit set for the weak $H^1$-topology of the orbit of $u_0$ consists of initial data that give rise to $H^1$-almost localized solutions.

\begin{prop}\label{limit_prop_2}
There exists $\widetilde{\varepsilon}>0$ small enough such that for every $u_0\in Y_+$ satisfying \eqref{pequeno} with $\varepsilon<\widetilde{\varepsilon}$ the following holds: For any strictly increasing sequence $t_n\to+\infty$ there exists a function $u_0^\star \in Y_+$, a subsequence $t_{\sigma(n)}$ and a $C^1$-function $x:\R\to\R$ satisfying \eqref{mod_rho_close}-\eqref{uniform_bound_param_mod} such that \begin{align}\label{prop_conv}
u\big(t_{\sigma(n)},\cdot+x(t_{\sigma(n)})\big)\rightharpoonup u_0^\star \,\hbox{ in }\, H^1 \ \ \hbox{ and }\ \ u\big(t_{\sigma(n)},\cdot+x(t_{\sigma(n)})\big)\to u_0^\star \,\hbox{ in }\, H^{1^-}_{loc}.
\end{align}
Moreover, the solution $u^\star(t)$ of equation \eqref{nov_eq_2} associated to $u_0^\star$ is $H^1$-almost localized.
\end{prop}

\begin{proof}
First of all notice that due to \eqref{pequeno}, \eqref{uniform_mod} and Lemma \ref{tech_lem_mon_exp} both inequalities \eqref{ineq_J_r} and \eqref{ineq_J_l} are satisfied by $u(t)$, the solution to \eqref{nov_eq_2} associated to $u_0$. Now, on the one hand due to \eqref{uniform_bound_param_mod} the family of functions $\{x(t_n+\cdot)-x(t_n)\}$ is uniformly equicontinuous, and hence by Arzela-Ascoli's Theorem we deduce the existence of a subsequence $\{t_{n_k}\}_{k\in\N}$ and a function $x^\star\in C(\R)$ such that, for all $T>0$ we have \begin{align}\label{mod_arz_asc}
x(t_{n_k}+\cdot)-x(t_{n_k})\to x^\star(\cdot) \,\hbox{ in }\, C([-T,T]).
\end{align}
Now we set $u_n(t):=u(t_{n}+t,\cdot+x(t_n+t))$. Notice that $u_n(t)$ defines a bounded sequence in $C(\R,H^1(\R))$ with $\{y_n\}_{n\in\N}$ bounded\footnote{This is a consequence of the fact that for all $t\in\R$ the momentum density belongs to $y(t)\in \mathcal{M}^+_b$ together with the fact that $u,u_x\in L^\infty(\R,L^\infty(\R))$.} in $L^\infty(\R,\mathcal{M}_{loc}^+)$ and hence there exists a function \[
u^\star\in L^\infty(\R,H^1(\R)) \ \hbox{ with }\ (1-\partial_x^2)u^\star\in L^\infty(\R,\mathcal{M}_{loc}^+),
\]
and a subsequence $\{u_{n_k},y_{n_k}\}$ such that \[
u_{n_k}\rightharpoonup u^\star \ \hbox{ in }\ L^\infty(\R,H^1(\R)) \ \ \hbox{ and }\ \ y_{n_k}\rightharpoonup (1-\partial_x^2)u^\star \ \hbox{ in }\ L^\infty(\R,\mathcal{M}^+)
\]
On the other hand, since $\{\partial_tu_{n_k}\}$ defines a bounded sequence in $L^\infty(\R,L^2(\R))$, Aubin-Lions' compactness Theorem ensures us that, up to a subsequence, we have \[
u_{n_k}\to u^\star\ \hbox{ a.e. in } \ \R^2 .
\]
Moreover, recalling that for all $t\in\R$, $\{\partial_xu_{n_k}(t)\}$ is bounded in $\mathrm{BV}_{loc}$, from Helly's selection Theorem we deduce \[
\partial_x u_{n_k}\to u^\star_x \ \hbox{ a.e. in }\ \R^2.
\]
Since $\{u_{n_k}\}$ and $\{\partial_x u_{n_k}\}$ are uniformly bounded on $\R$ we can pass to the limit on the Novikov equation \eqref{nov_eq_2} to deduce that $u^\star$ also satisfies the equation in the distributional sense. In particular, we deduce \[
u^\star_t\in L^\infty(\R,L^2(\R)) \ \hbox{ and therefore } \ u^\star\in C(\R,L^2(\R)).
\]
Finally, notice that $\{\partial_tu_{n_k}\}$ defines a bounded sequence in $L^\infty(\R,L^2(\R))$, and hence we have that for any $\phi\in C^\infty_c(\R)$, the map $t\mapsto \langle u_k,\phi\rangle$ defines a bounded uniformly equicontinuous sequence of continuous functions. Thus, by Arzela-Ascoli's Theorem and by density of $C^\infty_c(\R)$ in $H^1(\R)$ we obtain that for any $T>0$ \begin{align}\label{proof_conv_h1}
u_{n_k}\rightharpoonup u^\star \ \hbox{ in } \ C_{ti}([-T,T],H^1(\R)),
\end{align}
in particular, $u_{n_k}(0)\rightharpoonup u^\star(0)$. Now, for the sake of simplicity we split the proof in two steps. The first of them is devoted to prove the time-uniform exponential decay of $u^\star(t)$, while in the second one we intend to prove that $(1-\partial_x^2)u^\star\in L^\infty(\R,\mathcal{M}_b^+)$. 

\medskip

Notice that once we prove the latter property, and due to the fact that $Y_+(\R) \hookrightarrow H^{3/2^-}(\R) $, we immediately conclude that \[
u^\star\in L^\infty(\R,H^{\frac{3}{2}^-}(\R)),
\]
which, combined with $u^\star\in C(\R,L^2(\R))$, ensures us that $u^\star \in C(\R,H^1(\R))$. Therefore, $u^\star(t)$ belongs to the uniqueness class given in Theorem \ref{theorem_gwp}, and hence $u^\star(t)$ is the solution of the Novikov equation given by Theorem \ref{theorem_gwp} associated to $u^\star(0)$.

\medskip

\textbf{Step 1:} We claim that the limit function $u^\star(t)$ has time-uniform exponential decay. First of all, notice that since by \eqref{proof_conv_h1} for all $t\in\R$ we have \[
u(t_{n_k}+t,\cdot+x(t_{n_k}+t))\rightharpoonup u^\star(t,\cdot+x^\star(t)) \ \, \hbox{ in }\ \, H^1(\R),
\]
we deduce that it is enough to prove the claim at time $t=0$. On the other hand, since we have uniform bounds at the right (see Lemma \ref{tech_lem_mon_exp}), we immediately obtain the exponential decay of $u^\star_0$ at the right. In fact, let us consider the time sequence $t_{0,n_k}=t_{n_k}$ given by the convergence results at the beginning of this proof. Then, it is enough to notice that, by using the definition of $\Psi$ in \eqref{psi_def} and due to the fact that $\dot{x}(t)>(1+\epsilon)\dot{z}(t)$ for $t\in\R$ for some $\epsilon>0$,  by taking the limit $t_0\to+\infty$ we infer \[
\mathrm{I}_{t_0}^R(0)\to 0 \ \hbox{ as }\ t_0\to+\infty, \quad \hbox{and hence} \quad \limsup_{k\to+\infty}\,\mathrm{I}_{t_{0,n_k}}^R(t_{0,n_k})\leq e^{-R/K}.
\]
On the other hand, notice that by weak convergence in $H^1$ we deduce that for all $R\gg 1$ sufficiently large so that \eqref{L_infty_bound_outside} holds, we have
\[
\int \Big((u_0^\star)^2+(u_{0,x}^\star)^2\Big)(\cdot-x^\star(0))\Psi(\cdot-R)dx\leq \liminf_{k\to+\infty}\,\mathrm{I}_{t_{0,n_k}}^R(t_{0,n_k})\leq e^{-R/K}.
\]
Therefore, by Sobolev's embedding we conclude the exponential decay of $u_0^\star$ at the right of $x^\star(0)$.

\medskip

It only remains to prove the decay of $u_0^\star$ on the left. In fact, we shall prove the following property: There exist some constants $\mathbf{C}>0$, $\widetilde{c}>0$ and $R\gg1$ such that for all $A\geq r\geq R$ we have \[
\Vert u^\star_0\Vert_{H^1((-A,-r))}\leq \mathbf{C}e^{-\widetilde{c}r}.
\]
Notice that the latter inequality together with Sobolev's embedding implies the exponential decay of $u_0^\star$. We proceed by contradiction, that is, let us suppose that for all $\mathbf{C},\widetilde{c}>0$ and all $R\gg 1$ there exist $A\geq r\geq R$ sufficiently large such that \[
\Vert u_0^\star\Vert_{H^1((-A,-r))}\geq \mathbf{C}e^{-\widetilde{c}r}+\varepsilon, \quad \hbox{for some }\,\varepsilon>0.
\]
Thus, let us consider $1\gg\widetilde{c}>0$ small enough and $\mathbf{C}>0$ to be specified later. Notice that by the weak convergence result \eqref{proof_conv_h1} and due to the sequentially weakly lower-semicontinuity of the $H^1$-norm we have \[
\liminf_{k\to+\infty}\Vert u(t_{n_k},\cdot+x(t_{n_k}))\Vert_{H^1((-A,-r))}\geq\Vert u_0^\star\Vert_{H^1((-A,-r))}.
\]
Therefore, there exists $\widetilde{T}\gg 1$  and $k_0\geq 1$ sufficiently large such that for all $k\geq k_0$ we have \[
t_{n_k}\geq \widetilde{T} \quad \hbox{and} \quad \Vert u(t_{n_k},\cdot+x(t_{n_k}))\Vert_{H^1((-A,-r))}\geq \mathbf{C}e^{-\widetilde{c}r}+\tfrac{1}{2}\varepsilon.
\]
Now we consider a refinement of this subsequence which, for the sake of simplicity, we shall denote it by $\{t_n\}_{n\in\N}$, satisfying $t_n\geq \widetilde{T}$ for all $n\in\N$ and such that $x(t_{n+1})-x(t_{n})\geq 5(A+r)$. Then, by the almost monotonicity of the energy at the left (see \eqref{a_m_left_energy}) we obtain
\begin{align}
&\Vert u(t_{n+1},\cdot+x(t_{n+1}))\Vert_{H^1((-\infty,-r))}\geq \nonumber
\\ & \qquad \geq \Vert u(t_{n+1},\cdot+x(t_{n+1}))\Vert_{H^1((-A,-r))}+\Vert u(t_{n},\cdot+x(t_{n}))\Vert_{H^1((-\infty,-r))}-Ce^{-r/6}\nonumber
\\ & \qquad \geq \mathbf{C}e^{-\widetilde{c}r}+\tfrac{1}{2}\varepsilon+\mathbf{C}e^{-\widetilde{c}r}+\tfrac{\varepsilon}{2}-Ce^{-r/6}\geq \tfrac{19}{10}\mathbf{C}e^{-\widetilde{c}r}+\varepsilon,\label{exp_decay_proof}
\end{align}
where we are considering $\widetilde{c}$ and $\mathbf{C}$ such that $Ce^{-r/6}<\tfrac{1}{10}\mathbf{C}e^{-\widetilde{c}r}$. Notice that repeating the argument above, due to the fact that $x(t_{n+2})-x(t_{n+1})\geq 5(A+r)$, we can bound from below the $H^1$-norm on $(-\infty,-r)$ at time $t_{n+2}$ by the $H^1$-norm on $(-A,-r)$ at the same time plus the $H^1$-norm at time $t_{n+1}$ on $(-\infty,-r)$, plus some small error term. Therefore, by an iterative argument we conclude that $\Vert u(t_n,\cdot+x(t_n))\Vert_{H^1}\to +\infty$ as $n\to+\infty$, contradicting the energy conservation of the equation. Hence, we obtain that $u_0^\star$ has exponential decay. 

\medskip

\textbf{Step 2:} For the sake of simplicity, from now on we denote by $y^\star(t):=(1-\partial_x^2)u^\star(t)$. Now, we intend to prove that $y^\star\in L^\infty(\R,\mathcal{M}_b^+)$. First of all, notice that in the same fashion as in the proof of \textit{Step 1}, since for any $\phi\in C_c(\R)$ and any $t\in\R$ we have \[
\langle y(t_{n_k}+t,\cdot+x(t_{n_k}+t)),\phi\rangle\to \langle y^\star(t,\cdot+x^\star(t)),\phi\rangle,
\]
we deduce that it is enough to prove the claim at $t=0$. On the other hand, notice that for every compact set $K\subset \R$ and any $k\in\N$ we have \[
\Vert y(t_{n_k},\cdot+x(t_{n_k}))\Vert_{\mathcal{M}(K)}\leq C,
\]
where the constant $C>0$ only depends on $K$ and $\Vert u_0\Vert_{H^1}$. Therefore, by using Helly's selection Theorem we obtain that $u_{0,x}^\star\in \mathrm{BV}_{loc}$ and hence $y_0^\star$ is a positive Radon measure (locally finite possibly with infinite total mass on $\R$).  

\medskip

Now, we intend to take advantage of Lemma \ref{lem_measures_2} so that  we shall be able to estimate $y_0^\star$ by approximating $u_0^\star$ by a sequence of smooth functions. Hence, we define the approximating sequence \[
u_{0,m}^\star:=\rho_m*u_{0}^\star\in H^\infty(\R)\cap Y_+(\R), \ \, \hbox{ so that } \ \, y_{0,m}^\star\rightharpoonup y_0^\star \ \hbox{ in } \ \mathcal{M} .
\] 
We emphasize again that the previous weak convergence is in the sense of Definition \ref{def_weakly_conv}. Now, notice that due to the positivity of $y_{0,m}^\star$ on $\R$ and by using Young's inequality, recalling that $\Vert \rho_m\Vert_{L^1}=1$, we infer that for all $m\in\N$ we have \[
\int y_{0,m}^\star=\int u_{0,m}^\star\leq \Vert u_0^\star\Vert_{L^1}. 
\]
Hence, by the sequential weak lower semicontinuity given in Lemma \ref{lem_measures_2} we conclude \[
\Vert y_0^\star\Vert_{\mathcal{M}}\leq \liminf_{m\to+\infty}\Vert y_{0,m}^\star\Vert_{L^1}\leq \Vert u_0^\star\Vert_{L^1}.
\]
Therefore, $y_0^\star$ belongs to the space of finite Radon measures $\mathcal{M}_b^+$. The proof is complete.
\end{proof}

\subsection{Proof of Theorem \ref{MT1}} 

Let $\{t_n\}_{n\in\N}$ be any strictly increasing time sequence satisfying that $t_n\to+\infty$. Then, by the previous property we have that there exists a subsequence $\{t_{n_k}\}_{k\in\N}$ and an element $u_0^\star\in Y_+$ such that the solution associated to $u_0^\star$ is $H^1$-almost localized, and hence by Theorem \ref{MT2} we infer the existence of $x_0\in\R$ and $c^\star>0$ such that \[
u_0^\star=\varphi_{c^\star}(\cdot-x_0).
\] 

\textbf{Step 1:} Now, as we discussed before, once we know that the asymptotic object corresponds to a peakon, we are able to improve our local strong convergence result. In fact, due to the local strong $L^2$ convergence we deduce that for all $K\subset \R $ compact we have \[
\lim_{k\to+\infty}\Vert u(t_{n_k},\cdot+x(t_{n_k}))-\varphi_{c^\star}\Vert_{L^2(K)}=0.
\]
On the other hand, due to the fact that $\vert v_x\vert\leq v$ for any $v\in Y_+$ we deduce \[
\liminf_{k\to+\infty}\Vert u_x(t_{n_k},\cdot+x(t_{n_k}))\Vert_{L^2(K)}\leq \lim_{k\to+\infty}\Vert u(t_{n_k},\cdot+x(t_{n_k}))\Vert_{L^2(K)}=\Vert \varphi_{c^\star}\Vert_{L^2(K)}
\]
Hence, by using again that $\Vert \varphi'\Vert_{L^2(K)}=\Vert \varphi\Vert_{L^2(K)}$ we obtain \[
\liminf_{k\to+\infty}\Vert u(t_{n_k},\cdot+x(t_{n_k}))\Vert_{H^1(K)}^2\leq2\Vert \varphi_{c^\star}\Vert_{L^2(K)}^2=\Vert \varphi_{c^\star}\Vert_{H^1(K)}^2,
\]
Thus, by a standard result in Functional Analysis we know that the weak convergence result together with the previous inequality implies that \begin{align}\label{strong_h1}
u(t_{n_k},\cdot+x(t_{n_k}))-\varphi_{c^\star}\to0 \ \hbox{ in }\  H^1_{loc} \  \hbox{ as } \  k\to+\infty.
\end{align}

\textbf{Step 2:} Our aim now is to prove strong $H^1$ convergence in $(-A,\infty)$ for any fixed $A>0$. In fact, first of all, notice that the weak convergence result \eqref{prop_conv} together with the uniform estimate \eqref{uniform_mod} and the definition of $\varepsilon_*$ implies that \[
\Vert \varphi_{c^\star}(\cdot-x_0)-\varphi_c\Vert_{H^1}\leq C\varepsilon_* \quad \hbox{and} \quad 
\vert c-c^\star\vert\leq C\varepsilon_*\leq \dfrac{c}{2^9},
\]
and hence, by using the local strong convergence \eqref{strong_h1} we infer that $\vert x_0\vert \ll 1$. On the other hand, notice that the weak convergence result \eqref{prop_conv} forces $u_0^\star$ to satisfy the orthogonality condition \eqref{mod_rho_close}. Therefore, by using \eqref{orth_cond_def} we obtain that $x_0$ has to be equal to zero. Finally, notice that the convergence result \eqref{strong_h1} together with \eqref{uniform_mod} implies that \[
\sqrt{c^\star}=\lim_{k\to+\infty}\max_{\R}u(t_{n_k}).
\]
Thus, defining $\rho(t)=\max_\R u(t)$ we deduce that as $k\to+\infty$ we have \[
u(t_{n_k},\cdot+x(t_{n_k}))-\rho(t_{n_k})\varphi\rightharpoonup0 \ \hbox{ in }\ H^1.
\]
Since this is the only possible limit we conclude that as $t\to+\infty$ we have \begin{align}\label{weak_conv_H1_peakon}
u(t,\cdot+x(t))-\rho(t)\varphi\rightharpoonup 0 \  \hbox{ in } \ H^1 \ \ \hbox{ and } \ \ u(t,\cdot+x(t))-\rho(t)\varphi\to 0 \  \hbox{ in } \ H^{1}_{loc}.
\end{align}
Now, we claim that the latter convergence result implies that for any fixed $A>0$, as $t\to+\infty$, the following convergence holds: \begin{align}\label{local_strong_h1}
u(t,\cdot+x(t))-\rho(t)\varphi\to0  \ \hbox{ in } \ H^1((-A,\infty)).
\end{align}
In fact, let $\delta>0$ be fixed and consider $R\gg1$ sufficiently large such that \[
\mathcal{J}_r^R\big(u(0,\cdot+x(0)\big)<\delta \quad  \hbox{ and } \quad Ce^{-R/6}<\delta,
\] 
where $C>0$ is the constant involved in \eqref{ineq_J_r}. Then, from the almost decay of the energy at the right \eqref{ineq_J_r} we infer that \[
\mathcal{J}_r^R\big(u(t,\cdot+x(t))\big)<2\delta, \ \hbox{ for all } \, t\in\R.
\]
Nevertheless, the latter inequality together with the local strong convergence in $H^1$ given in \eqref{local_strong_h1} immediately implies that, for any $A>0$ we have \begin{align}\label{H1_conv_right}
u(t,\cdot+x(t))-\rho(t)\varphi\xrightarrow{t\to+\infty}0 \ \hbox{ in } \ H^1((-A,\infty)).
\end{align}

\textbf{Step 3:} Now we intend to prove that $\rho(t)\to \sqrt{c^\star}$ as $t\to+\infty$. In fact, let $\epsilon>0$ arbitrarily small but fixed and consider $R\gg1$ sufficiently large such that $Ce^{-R/6}<\epsilon$. Then, by using \eqref{ineq_J_l} as well as the energy conservation we obtain that for all $t>t'$ we have \[
\int \big(u^2+u_x^2\big)(t)\Psi(x-x(t)+R)\leq\epsilon+ \int \big(u^2+u_x^2\big)(t')\Psi(x-x(t')+R).
\]
On the other hand, due to the strong convergence result \eqref{local_strong_h1} and the exponential localization of both $\varphi$ and $\Psi$, we infer that there exists $t_0\gg1$ sufficiently large such that for all $t\geq t_0$ we have \[
\left\vert\int \big(u^2+u_x^2\big)(t)\Psi(x-x(t)+R)-\rho^2(t)E(\varphi)\right\vert\leq\epsilon.
\]
Plugging the last two inequalities together we conclude that for any pair of times $(t,t')\in\R^2$ satisfying $t>t'>T$ we have \[
\rho^2(t)E(\varphi)\leq \rho^2(t')E(\varphi)+3\epsilon.
\]
Since $\epsilon>0$ was arbitrary, the latter inequality forces $\rho(t)$ to have a limit at $+\infty$ and thus to converge to \[
\lim_{t\to+\infty}\rho(t)=\sqrt{c^\star}.
\]

\textbf{Step 4:} Now let us prove that $\dot{x}(t)\to c^\star$ as $t\to+\infty$. For the sake of readability let start by introducing some notation. Let $v,w,w_{n_0}:\R\to\R$ the functions given by  \[
v(t):=u-\sqrt{c^\star}\varphi(\cdot-x(t)), \quad w:=\sqrt{c^\star}\varphi(\cdot-x(t)) \  \hbox{ and } \ w_{n_0}:=\sqrt{c^\star}(\rho_{n_0}*\varphi)(\cdot-x(t)).
\]
Then, by differentiating the orthogonality condition in \eqref{mod_rho_close} and recalling that $\varphi$ satisfies the equation $\varphi-\varphi''=2\delta$ we obtain \[
\int v_tw_{n_0,x}=\dot{x}\int v(t,x)w_{n_0}(t,x)dx-2\dot{x}\sqrt{c^\star}\int v(t,x)\rho_{n_0}(x-x(t))dx.
\]
On the other hand, by using that $\varphi$ solves \eqref{nov_eq_2} we infer that $w(t,x)$ satisfy the following equation:
\[
w_t+(\dot{x}-c^\star)w_x+w^2w_x=p_x\Big(w^3+\dfrac{3}{2}ww_x^2\Big)-\dfrac{1}{2}p*w_x^3
\]
Therefore, by using that $u(t)$ also solves \eqref{nov_eq_2}, by replacing $u=v+w$ and then using the equation satisfied by $w$ we obtain
\begin{align}\label{mod_huge_eq}
v_t-(\dot{x}-c^\star)w_x&=-(v+w)^2v_x-(v^2+2vw)w_x-\dfrac{1}{2}p*\big(v_x^3+3v_x^2w_x+3v_xw_x^2\big)
\\ & \quad +p_x*\Big(v^3+3v^2w+3vw^2+\dfrac{3}{2}v(v_x+w_x)^2+\dfrac{3}{2}v_x^2w+2v_xww_x\Big).\nonumber
\end{align}
Now, notice that due to \eqref{H1_conv_right} and the exponential decay of $w$ and $w_{n_0}$ we infer that
\[
\Vert v^2w_{n_0,x}\Vert_{L^1}+\Vert v_x^2w_{n_0,x}\Vert_{L^1}+\int \vert vw_{n_0}\vert dx+\int \vert v(t,x)\rho_{n_0}(x-x(t))\vert dx\to0 \ \hbox{ as }\ t\to+\infty.
\]
Therefore, by taking the $L^2$-inner product from equation \eqref{mod_huge_eq} against $w_{n_0,x}$ and noticing that $\langle w_x(t),w_{n_0,x}(t)\rangle_{L^2,L^2}\equiv \mathrm{constant}>0$ for all times $t\in\R$ we conclude \[
\dot{x}-c^\star\to 0 \ \hbox{ as }\ t\to+\infty.
\]
\textbf{Step 5:} Now we intend to prove the strong $H^1$ convergence on $(\beta t,+\infty)$. In fact, let us start by recalling that from \eqref{H1_conv_right} we have that as $t\to+\infty$ the following convergence holds \[
u(t,\cdot)-\varphi_{c^\star}(\cdot-x(t))\rightharpoonup0 \, \hbox{ in }\, H^1(\R) \ \, \hbox{ and } \ \, u(t,\cdot+x(t))-\varphi_{c^\star}(\cdot)\to0 \, \hbox{ in }\, H^1((-A,\infty)).
\]
Now, let $\eta>0$ arbitrarily small but fixed. Let us consider $R\gg1$ sufficiently large such that \begin{align}\label{smallness_varphi_psi_proof}
\Vert \varphi\Vert_{H^1\left(\left(-\infty,-\frac{R}{2}\right)\right)}^2<\eta \quad \hbox{and}\quad \Vert \Psi-1\Vert_{L^\infty\left(\left(\frac{R}{2},+\infty\right)\right)}<\eta,
\end{align}
Thus, by the previous convergence results we infer the existence of a time point $t_0>0$ sufficiently large for which $x(t_0)>R$ and such that for all $t\geq t_0$ we have \[
\Vert u(t,\cdot+x(t))-\varphi_{c^\star}\Vert_{H^1\left(\left(-\frac{R}{2},+\infty\right)\right)}<\eta.
\]
On the other hand, by using \eqref{smallness_varphi_psi_proof} and the latter inequality we deduce that for all $r\geq R$ and all $t\geq t_0$ we have
\begin{align}\label{energy_bound}
\left\vert E(\varphi_{c^\star})-\int \left(u(t,\cdot+x(t))\varphi_{c^\star}+u_x(t,\cdot+x(t))\varphi_{c^\star}'\right)\Psi(\cdot+r)\right\vert\lesssim \eta.
\end{align}
From now on we consider $z(t)=\tfrac{1}{2}\beta t$. Notice that with this choice of $z(t)$ and due to the fact that $x(t)$ satisfies \eqref{uniform_bound_param_mod}, by straightforward computations we deduce that $z(t)$ satisfies the hypothesis of Lemma \ref{tech_lem_mon_exp} with $1-\delta=\tfrac{\beta}{4c}$ and $\gamma=\tfrac{1}{4}$. Moreover, as we discussed at the beginning of this section, $u(t)$ satisfies the corresponding hypothesis of Lemma \ref{tech_lem_mon_exp} for such choice of $\delta$. Hence, by using inequality \eqref{ineq_Itzero_It} we obtain that for all $t\geq t_0$ we have
\begin{align*}
&\int \left(u^2+u_x^2\right)(t,\cdot)\Psi\left(\cdot -x(t_0)-\tfrac{\beta}{2}(t-t_0)+R\right)
\\ & \qquad \quad \leq Ce^{-R/6} +\int \left(u^2+u_x^2\right)(t_0,\cdot)\Psi\left(\cdot -x(t_0)+R\right),
\end{align*}
where the constant $C$ now depends on $\delta$. Now, we define the variable $v(t):=u(t)-\sqrt{c^\star}\varphi(\cdot-x(t))$ and notice that 
\begin{align*}
&\int (v^2+v_x^2)(t,\cdot)\Psi\left(\cdot-x(t_0)-\tfrac{\beta}{2}(t-t_0)+R\right)
\\ & \qquad \quad =\int \left(u^2+u_x^2\right)(t,\cdot)\Psi\left(\cdot -x(t_0)-\tfrac{\beta}{2}(t-t_0)+R\right)
\\ & \qquad \qquad +c^\star\int \left(\varphi^2+\varphi_x^2\right)(t,\cdot-x(t))\Psi\left(\cdot -x(t_0)-\tfrac{\beta}{2}(t-t_0)+R\right)
\\ & \qquad \qquad -2\sqrt{c^\star}\int \left(u(t)\varphi(\cdot-x(t))+u_x(t)\varphi'(\cdot-x(t))\right)\Psi\left(\cdot -x(t_0)-\tfrac{\beta}{2}(t-t_0)+R\right)
\\ & \qquad \quad =:\mathrm{I}+\mathrm{II}+\mathrm{III}.
\end{align*}
On the other hand, notice that for all $t\geq t_0$ we have \[
x(t)-x(t_0)-\tfrac{\beta}{2}(t-t_0)+R\geq R.
\]
Hence, by using inequality \eqref{energy_bound} and then the exponential decay of $\varphi$ we infer that
\begin{align*}
\mathrm{I}+\mathrm{II}+\mathrm{III}&\leq\int \left(u^2+u_x^2\right)(t_0,\cdot)\Psi\left(\cdot -x(t_0)+R\right)+Ce^{-R/6}
\\ & \quad+c^\star\int \left(\varphi^2+\varphi_x^2\right)(t_0,\cdot-x(t_0))\Psi\left(\cdot -x(t_0)+R\right)+Ce^{-R/6}
\\ & \quad -2\sqrt{c^\star}\int \left(u(t_0)\varphi(\cdot-x(t_0))+u_x(t_0)\varphi'(\cdot-x(t_0))\right)\Psi\left(\cdot -x(t_0)+R\right)+C\eta
\\ & \lesssim \int \big(v^2+v_x^2)(t_0,\cdot)\Psi(\cdot-x(t_0)+R)+e^{-R/6}+\eta
\\ & \lesssim \eta+e^{-R/6},
\end{align*}
where we have used the exponential decay of $\varphi$ to obtain the latter inequality. Therefore, by taking $R\gg1$ sufficiently large and $t_1>t_0$ such that $\beta t_1\geq x(t_0)+\tfrac{\beta}{2}(t_1-t_0)-R$, we conclude that for all $t\geq t_1$ we have
\[
\int (v^2+v_x^2)(t,\cdot)\Psi\left(\cdot-\delta t\right)\lesssim \eta,
\]
which completes the proof the claim.

\medskip

\textbf{Step 6:} Finally, it only remains to prove the convergence in $(-\infty,z)$ for any $z\in\R$. This is a consequence of a more general property, noticed by Molinet in \cite{Mo3}, ensuring that all the energy of solutions associated to initial data in $Y_+$ is traveling to the right. In fact, we shall prove the following lemma which immediately conclude the proof of the theorem.
\begin{lem}\label{traveling_energy_lem}
For any $u_0\in Y_+$ and any $z\in\R$, the corresponding solution $u\in C(\R,H^1(\R))$ to equation \eqref{nov_eq_2} associated $u_0$ satisfies \begin{align*}
\lim_{t\to+\infty}\Vert u(t)\Vert_{H^1((-\infty,z))}=0.
\end{align*}
\end{lem}

\begin{proof}[Proof of Lemma \ref{traveling_energy_lem}]
First of all notice that, for $\Psi$ defined in \eqref{psi_def_2}, for any time $t\in\R$ fixed the map \[
z\mapsto\int \big(u^2+u_x^2\big)(t,x)\Psi(\cdot-z)dx,
\]
defines a decreasing continuous bijection from $\R$ into $(0,\Vert u_0\Vert_{H^1}^2)$. Therefore, by setting any $0<\gamma<\Vert u_0\Vert_{H^1}^2$, we deduce that the map $x_\gamma:\R\to\R$ defined by the equation \begin{align}\label{def_xgamma}
\int \big(u^2+u_x^2\big)(t,x)\Psi(\cdot-x_\gamma(t))dx=\gamma,
\end{align}
is well-defined. Moreover, recalling that $u\in C(\R,H^1(\R))$ we infer from \eqref{def_xgamma} that $x_\gamma$ is a continuous function. Now, notice that in order to conclude the proof of the lemma it is enough to show that for any $\gamma\in(0,\Vert u_0\Vert_{H^1}^2)$ we have \begin{align}\label{lim_final}
\lim_{t\to+\infty} x_\gamma(t)=+\infty.
\end{align}
For the sake of readability we split the proof of the latter property in two steps.

\medskip

\textbf{Step 1:} First we claim that for any $\Delta>0$ and any $t\in\R$ we have \begin{align}\label{claim_step}
x_\gamma(t+\Delta)-x_\gamma(t)\geq \dfrac{2}{5}\int_t^{t+\Delta}\int u^2(t,x)\Psi'(\cdot-x_\gamma(t))dx >0.
\end{align}
First of all, notice that by continuity with respect to the initial data it is enough to prove the claim for solutions $u\in C^\infty(\R,H^\infty(\R))\cap L^\infty(\R,H^1(\R))$. On the other hand, as an application of the Implicit Function Theorem we obtain that $x_\gamma(t)$ is of class $C^1$. In fact, let us define the functional \[
\psi(v,z):=\int \big(v^2+v_x^2\big)\Psi(\cdot-z)dx.
\]
Notice that $\psi$ clearly defines a $C^1$ function on $H^1(\R)\times \R$. Moreover, notice that since any function $v\in Y_+\setminus\{0\}$ cannot vanish at any point $x\in\R$, we deduce that for any function $v\in H^\infty\cap Y_+$ and any $z\in\R$ we have \[
\dfrac{\partial\psi}{\partial z}=\int \big(v^2+v_x^2\big)\Psi'(\cdot-z)>0.
\]
Recalling equation \eqref{dt_I_J_i} from the proof of Lemma \ref{tech_lem_mon_exp}, we obtain 
\begin{align*}
\dot{x}_\gamma\int \big(u^2+u_x^2\big)\Psi'(\cdot-x_\gamma)&=\int u^2u_x^2\Psi'+\int \{p*(3uu_x^2+2u^3)\}u\Psi'+\int \{p_x*u_x^3\}u\Psi'.
\end{align*}
Now, due to the fact that $\vert v_x\vert\leq v$ for any $v\in Y_+$ we deduce $p*uu_x^2+p_x*u_x^3\geq 0$. On the other hand, since $u(t)$ is positive, from Lemma \ref{tech_ineq_lem} we infer
\[
p*(3uu_x^2+5u^3)\geq 2u^3 \ \hbox{ in particular } \ p*(2uu_x^2+2u^3)\geq \tfrac{4}{5}u^3.
\]
Hence, by using again that $\vert v_x\vert\leq v$ for any $v\in Y_+$ and the previous inequalities we infer that
\begin{align*}
2\dot{x}_\gamma\int u^2\Psi'(\cdot-x_\gamma)&\geq \int u^2u_x^2\Psi'+\dfrac{4}{5}\int u^4\Psi'.
\end{align*}
Therefore, due to the fact that $\Psi'$ is a non-negative function with $\Vert \Psi'\Vert_{L^1}=1$, by using H\"older's inequality we obtain
\begin{align*}
\dot{x}_\gamma(t)\geq \dfrac{2}{5}\int u^2\Psi'(\cdot-x_\gamma(t))dx.
\end{align*}
Integrating in time between $t$ and $t+\Delta$ we conclude the claim.

\medskip

\textbf{Step 2:} Finally, in this last step we intend to conclude the proof of \eqref{lim_final}. First of all notice that from the claim of the previous step we deduce, in particular, that $x_\gamma(\cdot)$ is increasing and hence it has a limit $x_\gamma^\infty\in \R\cup\{+\infty\}$, i.e. \begin{align*}
\lim_{t\to+\infty}x_\gamma(t)=x_\gamma^\infty.
\end{align*}
Therefore, the proof of \eqref{lim_final} is equivalent to prove that $x_\gamma^\infty=+\infty$. We proceed by contradiction, i.e. let us suppose that $x_\gamma^\infty\in\R$. Thus, this fact together with inequality \eqref{def_xgamma} and the fact that $\vert u_x\vert\leq u\leq \Vert u_0\Vert_{H^1}$ for all $(t,x)\in\R^2$ implies 
\begin{align}\label{contrad_2}
\lim_{t\to+\infty}\int\big(u^2+u_x^2\big)(t,x)\Psi(\cdot-x_\gamma(t))=\lim_{t\to+\infty}\int \big(u^2+u_x^2\big)(t,x)\Psi(\cdot-x_\gamma^\infty)=\gamma.
\end{align}
On the other hand, by taking $\Delta=1$, from \eqref{claim_step} and the convergence of $x_\gamma(t)$ we obtain 
\[
\lim_{t\to+\infty}\int_{t}^{t+1}\int u^2\Psi'(\cdot-x_\gamma(t))=\lim_{t\to+\infty}\int_t^{t+1}\int u^2\Psi'(\cdot-x_\gamma^\infty)=0.
\]
Notice that the latter equality together with the fact that $\vert v_x\vert\leq v$ for any $v\in Y_+$ implies, in particular, that there exists a sequence of times $t_n\to+\infty$ such that for any compact set $K\subset\R$ the following holds:
\begin{align}\label{zero_compact}
\lim_{n\to+\infty}\Vert u(t_n)\Vert_{L^\infty(K)}=0.
\end{align}
Now we choose any $\gamma<\gamma'<\Vert u_0\Vert_{H^1}$, arbitrary but fixed. Then, we consider the compact set \[
K:=[x^\infty_\gamma-M,x_\gamma^\infty+M],
\]
with $M\gg1$ sufficiently large such that $x_\gamma^\infty-M<x_{\gamma'}(0)$. Hence, by using \eqref{zero_compact}, the monotonicity of $t\mapsto x_{\gamma'}(t)$ and recalling that $x_{\gamma'}(0)<x_{\gamma}(0)$ we conclude \[
\lim_{n\to+\infty}\int \big(u^2+u_x^2\big)(t_n,x)\Psi(\cdot-x_{\gamma}^\infty)=\gamma'.
\]
However, this contradicts hypothesis \eqref{contrad_2} and hence the proof is complete.
\end{proof}

\section{Appendix}

\subsection{Proof of Lemma \ref{tech_lem_mon_exp}}\label{tech_lem_appendix}

The following computations can be made rigorously by standard approximation and density arguments by considering, for instance, the convolution of $u_0$ with the mollifiers family $\rho_n$ defined in \eqref{def_rho} and by using the second statement in Theorem \ref{theorem_gwp}. We refer to \cite{EM1} for a complete justification of this argument. 

\medskip

Our aim is to prove the first inequality in \eqref{ineq_Itzero_It} by integrating its time derivative. In fact, by direct differentiation from the definition of $\mathrm{I}_{t_0}(t)$ we obtain
\begin{align}
\dfrac{d}{dt}\mathtt{I}_{t_0}^R(t)&=2\int \big(uu_t+u_xu_{xt}\big)\Psi-\dot z(t)\int (u^2(t)+u_x^2(t)\big)\Psi'\nonumber
\\ & =:\mathrm{J}-\dot z(t)\int (u^2(t)+u_x^2(t)\big)\Psi'.\label{dt_I_thm}
\end{align}
On the other hand, by using both equations \eqref{novikov_eq} and \eqref{nov_eq_2}, after integration by parts we get
\begin{align*}
\mathrm{J}&=2\int \big(u_t-u_{txx}\big)u\Psi-2\int uu_{tx}\Psi'
\\ & =2\int \big(3uu_xu_{xx}+u^2u_{xxx}-4u^2u_x\big)u\Psi
\\ & \quad +2\int (u^2u_{xx}+2uu_x^2+p_x*(3uu_xu_{xx}+2u_x^3+3u^2u_x))u\Psi'
\\ & =4\int u^2u_x^2\Psi'+2\int u^4\Psi'+2\int \{p_x*(3uu_xu_{xx}+2u_x^3+3u^2u_x)\}u\Psi'.
\end{align*}
On the other hand, recalling that for any $L^2$ function $f:\R\to\R$ we have $p*f_x=p_x*f$, and by using that $p$ is the fundamental solution of $(1-\partial_x^2)$, we obtain
\begin{align*}
2p_x*(3uu_xu_{xx}+2u_x^3+3u^2u_x)=-2u^3-3uu_x^2+3p*uu_x^2+2p*u^3+p_x*u_x^3.
\end{align*}
Hence, by plugging this into \eqref{dt_I_thm} we get \begin{align}
\dfrac{d}{dt}\mathtt{I}_{t_0}^R(t)&=-\dot z(t)\int \big(u^2+u_x^2\big)\Psi'+\int u^2u_x^2\Psi'+\int \{p*(3uu_x^2+2u^3)\}u\Psi'+\int \{p_x*u_x^3\}u\Psi'\nonumber
\\ & =-\dot z(t)\int \big(u^2+u_x^2\big)\Psi'+\mathrm{J}_1+\mathrm{J}_2+\mathrm{J}_3.\label{dt_I_J_i} 
\end{align}
In order to bound $\mathrm{J}_i$, for $i=1,2,3$, we split $\R$ into two complementary regions related to the size of $u(t)$. In fact, we start by rewriting $\mathrm{J}_1$ as \[
\mathrm{J}_1=\int_{\vert x-x(t)\vert<R_0} u^2u_x^2\Psi'+\int_{\vert x-x(t)\vert>R_0} u^2u_x^2\Psi'=:\mathrm{J}_1^1+\mathrm{J}_1^2
\]
Now notice that \eqref{hip_z_dot} ensures that $\dot{x}(t)-\dot{z}(t)\geq \gamma c$ for all $t\in\R$, and hence by using the definition of $z_{t_0}^R(t)$ we deduce that for $\vert x-x(t)\vert <R_0$ we have 
\begin{align}\label{mod_distance}
x-z_{t_0}^R(t)=x-x(t_0)- R-z(t)+z(t_0)\leq R_0-R-\gamma c(t_0-t).
\end{align}
Therefore, due to the decay property of $\Psi'$, H\"older's inequality and by using Sobolev's embedding together with the conservation of the $H^1$-norm we obtain \[
\mathrm{J}_1^1\lesssim \Vert u_0\Vert_{H^1}^4e^{\frac{1}{6}(R_0-R-\gamma c(t_0-t))}.
\]
On the other hand, by using \eqref{L_infty_bound_outside} we infer that for all $t\leq t_0$ we have \[
\mathrm{J}_1^2\lesssim\Vert u\Vert_{L^\infty(\{\vert x-x(t)\vert\geq R_0\})}^2\int u_x^2\Psi'\lesssim \dfrac{(1-\delta)c}{2^6}\int u_x^2\Psi'.
\]
Thus, the latter integral can be absorbed by the first integral term of \eqref{dt_I_J_i}. Now, in order to bound $\mathrm{J}_2$ we proceed similarly by decomposing the space into two regions related to the size of $u(t)$. In fact, rewriting $\mathrm{J}_2$ as 
\[
\mathrm{J}_2=\int_{\vert x-x(t)\vert<R_0}\{p*(3uu_x^2+2u^3)\}u\Psi'+\int_{\vert x-x(t)\vert>R_0}\{p*(3uu_x^2+2u^3)\}u\Psi'=:\mathrm{J}_2^1+\mathrm{J}_2^2.
\]
Now, we recall that by standard Fourier analysis the operator 
$(1-\partial_x^2)^{-1}:L^2(\R)\to H^2(\R)$ is continuous. Thus, due to the decay property of $\Psi'$, by using Sobolev's embedding together with the conservation of the $H^1$-norm we obtain\[
\mathrm{J}_2^1=\int_{\vert x-x(t)\vert<R_0}\{p*(3uu_x^2+2u^3)\}u\Psi'\lesssim \Vert u_0\Vert_{H^1}^4e^{\frac{1}{6}(R_0-R-\gamma c(t_0-t))}.
\]
On the other hand, notice that since $\Psi'$ is positive and due to \eqref{psi_bound} we have
\begin{align}\label{trick_psi_prime}
(1-\partial_x^2)\Psi'\geq \dfrac{1}{2}\Psi', \quad \hbox{what implies that}\quad (1-\partial_x^2)^{-1}\Psi'\leq 2\Psi',
\end{align}
and hence, by using the previous inequality together with \eqref{L_infty_bound_outside}, we infer that \begin{align*}
\mathrm{J}_2^2&\lesssim \Vert u_0\Vert_{H^1(\vert x-x(t)\vert>R_0)}\Vert u_0\Vert_{H^1}\int u^2\Psi'
\lesssim \dfrac{(1-\delta)c}{2^6}\int u^2\Psi'.
\end{align*}
Therefore, we can absorb this term by the first integral on the right-hand side of \eqref{dt_I_J_i}, and hence it only remains to bound $\mathrm{J}_3$. Now, in order to bound it we proceed again by decomposing the space into two regions \[
\mathrm{J}_3=\int_{\vert x-x(t)\vert<R_0} \{p_x*u_x^3\}u\Psi'+\int_{\vert x-x(t)\vert>R_0} \{p_x*u_x^3\}u\Psi'=:\mathrm{J}_3^1+\mathrm{J}_3^2.
\]
Thus, due to the decay property of $\Psi'$, by using Sobolev's embedding together with the conservation of the $H^1$-norm we obtain
\[
\mathrm{J}_3^1:=2\int_{\vert x-x(t)\vert<R_0}\{p_x*u_x^3\}u\Psi'\lesssim \Vert u_0\Vert_{H^1}^4e^{\frac{1}{6}(R_0-R-\gamma c(t_0-t))}.
\]
On the other hand, by using \eqref{trick_psi_prime} and \eqref{L_infty_bound_outside} again, we obtain
\[
\mathrm{J}_3^2=2\int_{\vert x-x(t)\vert>R_0}\{p_x*u_x^3\}u\Psi'\lesssim \Vert u_0\Vert_{H^1(\vert x-x(t)\vert>R_0)}\Vert u_0\Vert_{H^1}\int u_x^2\Psi'\lesssim \dfrac{(1-\delta)c}{2^6}\int u_x^2\Psi'.
\]
Thus we can absorb this term by the first integral on the right-hand side of \eqref{dt_I_J_i}. Therefore, gathering all of these inequalities we infer that for $R>R_0$ there exists $C>0$ only depending on $\delta$, $\gamma$, $c$, $R_0$ and $\Vert u_0\Vert_{H^1}$ such that for all $t\leq t_0$ it holds \[
\dfrac{d}{dt}\mathrm{I}_{t_0}(t)\leq Ce^{-\frac{1}{6}(R+\gamma c(t_0-t))}-\dfrac{\dot{z}(t)}{2}\int \big(u^2+u_x^2\big)\Psi'.
\]
Finally, integrating between $t$ and $t_0$ we deduce that for all $t\leq t_0$ \[
\mathrm{I}_{t_0}(t_0)-\mathrm{I}_{t_0}(t)\leq Ce^{-\frac{R}{6}}.
\]
Notice that the second inequality in \eqref{ineq_Itzero_It} is obtained in exactly the same fashion, except for obvious modifications, so we omit it. The proof is complete. \qed

\medskip

\subsection{Proof of Proposition \ref{MT3}}\label{proof_MT_appendix}
First of all notice that in order to conclude the proof it is enough to show that $\mathrm{I}_{t_0}(t)\to 0$ as $t\to-\infty$. In fact, by plugging this limit property into the first inequality in \eqref{ineq_Itzero_It} we obtain \begin{align}\label{claim_I_exp}
\mathrm{I}_{t_0}(t_0)\leq Ce^{-\frac{R}{6}}.
\end{align}
Hence, let us prove that the limit equality holds. In fact, we start by spliting the space as before. In concrete, for $R_\varepsilon\gg1$ to be specified later, we decompose $\mathrm{I}_{t_0}(t)$ into  \[
\mathrm{I}_{t_0}(t)=\int_{\vert x-x(t)\vert<R_\varepsilon}\big(u^2+u_x^2\big)\Psi+\int_{\vert x-x(t)\vert>R_\varepsilon}\big(u^2+u_x^2\big)\Psi=\mathrm{I}_1+\mathrm{I}_2.
\] 
Now, on the one-hand, by the $H^1$-almost localized hypothesis, for any $\varepsilon>0$ we can choose $R_\varepsilon$ sufficiently large such that $\mathrm{I}_2<\tfrac{\varepsilon}{2}$. On the other hand, by monotonicity of $\Psi$ and by Sobolev's embedding we get
\begin{align}\label{I_1_ineq}
\mathrm{I}_1\lesssim\Vert u_0\Vert_{H^1}^2\Psi\big(R_\varepsilon+x(t)-x(t_0)-R-z(t)+z(t_0)\big).
\end{align}
Finally, recalling that $\dot x(t)-\dot z(t)\geq\gamma c$ for all $t\in\R$, we deduce that for $\vert x-x(t)\vert<R_\varepsilon$ we have \begin{align*}
x-z_{t_0}^R(t)=x-x(t_0)- R-z(t)+z(t_0)\leq R_\varepsilon-R-\gamma c(t_0-t).
\end{align*} 
Thus, by plugging the latter inequality into \eqref{I_1_ineq} and due to the fact that $\Psi(x)\to 0$ as $x\to -\infty$ we obtain that $\mathrm{I}_1\to0$ as $t\to-\infty$. Finally, notice that gathering \eqref{energy_right} with \eqref{claim_I_exp} we infer that for any $x_0\gg1$ sufficiently large and all $t\in\R$ it holds \[
\Vert u(t,\cdot+x(t))\Vert_{H^1(x_0,+\infty)}\leq Ce^{-\frac{x_0}{6}}.
\]
On the other hand, since the Novikov equation is invariant under space-time inversion, that is, invariant under the transformation $(t,x)\mapsto(-t,-x)$, the latter inequality also leads us to  \[
\Vert u(t,\cdot+x(t))\Vert_{H^1(-\infty,x_0)}\leq Ce^{-\frac{x_0}{6}}.
\]
Therefore, we conclude the proof by using Sobolev's embedding. The proof is complete. \qed

\medskip

\subsection{Proof of Lemma \ref{integral_line_nov}}\label{appendix_integral_line}

First of all notice that \eqref{equality_supp} follows directly from combining formulas \eqref{positive_mom_1}-\eqref{positive_mom_2} with \eqref{supp_y_prop_nov}. On the other hand, notice that the remaining part of the Lemma would follow directly from the definition of $x_+(\cdot)$ and $q(\cdot,\cdot)$ if the initial data were in $H^3(\R)$, and hence  we shall proceed by approximating the solution at some convenient time by smooth functions as before. Moreover, the proof follows by contradiction, i.e., from now on we assume that there exists $t^*\in\R$ such that \[
q(t^*,x(0)+x_+(0))=x(t^*)+x_+(t^*)+\varepsilon,
\]
for some $\varepsilon\neq 0$. Notice also that without loss of generality we can assume that $t^*\in(0,1)$. We split the proof in two cases regarding the sign of $\varepsilon$. \begin{enumerate}
\item \textbf{Case $\varepsilon<0$:} In this case we approximate the initial data $u_0$ by the family of smooth functions \[
u_{0,n}:=\rho_n*u_0\in H^\infty(\R)\cap Y_+(\R).
\] 
Now, by continuity and monotonicity of the map $x\mapsto q(t,x)$ we deduce that there exists $\delta>0$ such that \[
q(t^*,x(0)+x_+(0)+\delta)<x(t^*)+x_+(t^*)+\tfrac{1}{2}\varepsilon,
\]
On the other hand, notice that by definition of $\rho_n$, there exists $n_0\in\N$ sufficiently large such that for all $n\geq n_0$ we have \[
y_{0,n}\equiv 0 \,\hbox{ on }\,[x(0)+x_+(0)+\delta,\infty).
\]
Thus, denoting by $u_n(t)$ the solution to \eqref{nov_eq_2} associated to $u_{0,n}$, we consider the characteristic $q_n:\R\to\R$ defined by \[
\begin{cases}
\tfrac{d}{dt}q_n(t)=u_n^2\big(t,q_n(t,x)\big), & t\in\R,
\\ q_n(0)=x(0)+x_+(0)+\delta.
\end{cases}
\]
It is clear from the definition that $ y_n(0,\cdot)\equiv 0$ on $[q_n(0),+\infty)$. Therefore, by using formula \eqref{eq_flow} we obtain that \[
y_n(t^*,\cdot)\equiv0 \,\hbox{ on }\,[q_n(t^*),+\infty).
\] 
Finally, since $q_n(\cdot)\to q(\cdot,x(0)+x_+(0)+\delta)$ in $C([0,1])$ and by using \eqref{convergence_h1_ti} we conclude that, for $n\in\N$ sufficiently large, \[
y(t^*,\cdot )\equiv 0 \,\hbox{ on }\, [x(t^*)+x_+(t^*)+\tfrac{1}{4}\varepsilon,+\infty),
\]
what contradicts the definition of $x_+(t^*)$ due to the fact that $\varepsilon<0$. The proof of this case is complete.
\item \textbf{Case $\varepsilon>0$:}  In this case we approximate the solution at time $t^*$ by the family of smooth functions \[
u^*_{n}:=\rho_n*\big(u(t^*)\big)\in H^\infty(\R)\cap Y_+(\R).
\] 
Now, by continuity and monotonicity of the map $x\mapsto q(t,x)$ we deduce that there exists $\delta>0$ such that \[
q(t^*,x(0)+x_+(0)-\delta)>x(t^*)+x_+(t^*)+\tfrac{1}{2}\varepsilon,
\]
On the other hand, by denoting $\widetilde{u}_n(t)$ the solution to \eqref{nov_eq_2} such that $\widetilde{u}_n(t^*)=u_n^*$, we deduce by definition of $\rho_n$ that there exists $n_0\in\N$ sufficiently large such that, for all $n\geq n_0$ we have \[
\widetilde{y}_{n}(t^*)\equiv 0 \,\hbox{ on }\,\left[ q\big(t^*,x(0)+x_+(0)-\delta\big),\infty\right).
\]
Thus, as before, we consider the characteristic $q_n:\R\to\R$ defined by \[
\begin{cases}
\tfrac{d}{dt}q_n(t)=\widetilde{u}_n^2\big(t,q_n(t,x)\big), & t\in\R,
\\ q_n(t^*)=q(t^*,x(0)+x_+(0)-\delta).
\end{cases}
\]
It is clear from the definition that $ \widetilde{y}_n(t^*,\cdot)\equiv 0$ on $[q_n(t^*),+\infty)$. Hence, by using formula \eqref{eq_flow} we obtain that \[
\widetilde{y}_n(0,\cdot)\equiv0 \,\hbox{ on }\,[q_n(0),+\infty).
\] 
In the same fashion as before, by using \eqref{convergence_h1_ti} we conclude that, for $n\in\N$ sufficiently large we have \[
y(0,\cdot )\equiv 0 \,\hbox{ on }\, [x(0)+x_+(0)-\tfrac{1}{4}\delta,+\infty),
\]
what contradicts the definition of $x_+(0)$. The proof of this case is complete.
\end{enumerate}
Therefore, the proof is complete. \qed

\medskip

\subsection{Proof of Lemma \ref{tech_ineq_lem}}\label{tech_ineq_lem_appendix}

\textbf{Step 1:} Let us first prove inequality \eqref{point_conv}. In fact, let $x\in\R$ arbitrary but fixed. Then, on the one-hand, by using $a^2+b^2\geq 2ab$ we have 
\begin{align*}
e^{-x}\int_{-\infty}^xe^\eta\big(vv_x^2(\eta)+v^3(\eta)\big)d\eta&\geq 2e^{-x}\int_{-\infty}^xe^\eta v^2(\eta)v_x(\eta)d\eta
\\ & =\dfrac{2}{3}v^3(x)-\dfrac{2e^{-x}}{3}\int_{-\infty}^xe^\eta v^3(\eta)d\eta,
\end{align*}
and hence,
\begin{align*}
e^{-x}\int_{-\infty}^xe^\eta\left(vv_x^2(\eta)+\dfrac{5}{3}v^3(\eta)\right)d\eta&\geq \dfrac{2}{3}v^3(x).
\end{align*}
On the other hand, by using $a^2+b^2\geq -2ab$ we obtain \begin{align*}
e^{x}\int_x^{+\infty}e^{-\eta}\big(vv_x^2(\eta)+v^3(\eta)\big)d\eta&\geq 2e^{x}\int_{x}^{+\infty}e^{-\eta} v^2(\eta)v_x(\eta)d\eta
\\ & =\dfrac{2}{3}v^3(x)-\dfrac{2e^{x}}{3}\int_{x}^{+\infty}e^{-\eta} v^3(\eta)d\eta,
\end{align*}
and hence \begin{align*}
e^{x}\int_x^{+\infty}e^{-\eta}\big(vv_x^2(\eta)+\dfrac{5}{3} v^3(\eta)\big)d\eta\geq \dfrac{2}{3}v^3(x).
\end{align*}
Gathering both inequalities we conclude \eqref{point_conv}.\qed

\medskip

\textbf{Step 2:} Let us prove now the second part of the statement. In fact, from the proof of the first step we see that equality holds if and only if \[
vv_x^2+v^3=2v^2v_x \,\hbox{ a.e. on }\, (-\infty,x_0)\quad \hbox{or equivalently} \quad  v_x=v \,\hbox{ a.e. on}\, (-\infty,x_0).
\]
By solving the ODE and by continuity of $v$ this forces it to be $v(z)=Ce^z$ on $(-\infty,x_0)$. In the same fashion, equality holds in \eqref{point_conv} if and only if \[
vv_x^2+v^3=-2v^2v_x \,\hbox{ a.e. on }\, (x_0,+\infty)\quad \hbox{or equivalently} \quad  v_x=v \,\hbox{ a.e. on}\, (x_0,\infty
).
\] 
Thus, by solving the ODE we obtain that $v(z)=Be^{-z}$ on $(x_0,+\infty)$. Therefore, by continuity of $v$ in $\R$ we conclude that $v(x)=Ce^{-\vert x-x_0\vert}$. The proof is complete. \qed

\medskip

\subsection{Proof of Lemma \ref{modulational_lemma}}\label{apendix_mod}
We shall follow Molinet's proof \cite{Mo}. Let $n_0\in\N$ to be specified. Consider the functional given by the ortogonality condition we are looking for, that is, consider the functional given by \begin{align*}
Y_z(u,y):=\int u(\rho_{n_0}*\varphi)'(x-y-z).
\end{align*}
Notice that $Y_z:H^1(\R)\times \R\to \R$ 
defines a $\mathcal{C}^1$ functional in a neighborhood of $(\varepsilon,0)$. Moreover, since by definition both $\rho_{n_0}$ and $\varphi$ are even functions, we have $Y_z(0,\varphi(\cdot-z))\equiv 0$. On the other hand, notice that by direct computations we have
\begin{align}\label{derivative_funct_mod_proof}
\dfrac{\partial Y_z}{\partial y}\big(\varphi(\cdot-z),0\big)=\int \varphi'(\rho_{n_0}*\varphi')=\Vert \varphi'\Vert_{L^2}^2-\varepsilon(n_0)=1-\varepsilon(n_0),
\end{align}
where $\varepsilon(\cdot)$ satisfy $\varepsilon(n)\to 0$ as $n\to+\infty$. Hence, we are able to take $n_0\in\N$ large enough such that \[
\dfrac{\partial Y_z}{\partial y}\big(\varphi(\cdot-z),0\big)\geq \dfrac{1}{2}.
\]
Thus, from the Implicit Function Theorem we deduce the existence of $\epsilon>0$, $\delta>0$ small enough and a $\mathcal{C}^1$ function $y_z(\cdot):B(\varphi(\cdot-z),\epsilon)\to(-\delta,\delta)$ such that
\[
Y_z(u,y_z(u))=0, \quad \hbox{for all } \ u\in B(\varphi(\cdot-z),\epsilon),
\]
where $B(\varphi(\cdot-z),\epsilon)$ denotes the $H^1$-ball of radius $\epsilon$ centered at $\varphi(\cdot-z)$.
In particular, as a consequence of the Implicit Function Theorem we deduce the existence of a constant $C_0>0$ such that for any $\widetilde{\epsilon}\leq\epsilon$ we have \begin{align}\label{ift_consequences}
\hbox{if }\  u\in B(\varphi(\cdot-z),\widetilde{\epsilon}) \ \,\hbox{ then }\ \, \vert y_z(u)\vert\leq C_0\widetilde{\epsilon}.
\end{align}
Notice that by a translation invariance argument, the constants $\epsilon$, $\delta$ and $C_0$ are independent of $z\in\R$. Therefore, by uniqueness we can define a $C^1$ map 
\[
\widetilde{x}: \bigcup_{z\in\R}B\big(\varphi(\cdot-z),\widetilde{\varepsilon}_0\big)\to(-\delta,\delta) \quad \hbox{given by} \quad \widetilde{x}:=z+y_z(u) \, \hbox{ for } \, u\in B(\varphi(\cdot-z),\epsilon).
\]
On the other hand, notice that $Y_z$ is also of class $C^1$ viewed as a functional $Y_z:L^2(\R)\times \R\to\R$. Moreover, by the same computations as before we have \[
\dfrac{\partial Y_z}{\partial y}(\varphi(\cdot-z),y)=\int \varphi(x)(\rho_{n_0}''*\varphi)(x-z-y)dx 
\]
Hence, arguing in the same fashion as before we deduce the existence of a constant $\widehat{\epsilon}>0$ and a $C^1$ function \begin{align}\label{functional_l21}
\widehat{x}:\bigcup_{z\in\R}B(\varphi(\cdot-z),\eta)\to(-\kappa,\kappa),
\end{align}
such that for all $u\in \mathcal{B}$ we have  $Y_z(\varphi(\cdot-z-),y_z(u))=0$ if and only if $y_z(u)=\widehat{x}(u)$, where the set $\mathcal{B}$ is given by \begin{align*}
\mathcal{B}:=\bigcup_{z\in\R}B(\varphi(\cdot-z),\eta).
\end{align*}
Notice that in both \eqref{functional_l21} and the definition of $\mathcal{B}$ we are denoting by $B(\varphi(\cdot-z),\eta)$ the $L^2$-ball centered at $\varphi(\cdot-z)$. Thus, by setting $\varepsilon^*:=\min\{\epsilon,\eta\}$, due to the uniqueness given by the Implicit Function Theorem we conclude that $\widetilde{x}=\widehat{x}$ on $B_{H^1}(\varphi(\cdot-z),\varepsilon^*)$. Hence, $\widetilde{x}(\cdot)$ is also a $C^1$ function on $\cup_{z\in\R}B_{H^1}(\varphi(\cdot-z),\varepsilon^*)$ endowed with the $L^2$ norm. This fact shall be important in the sequel. Now, for the sake of readability we split the proof in several steps.

\medskip

\textbf{Step 1:} Our first aim is to prove \eqref{mod_rho_close}. In fact,  notice that as a direct consequence of \eqref{hip_lem_mod} we have \[
\left\{\tfrac{1}{\sqrt{c}}u(t,\cdot): \ t\in\R\right\}\subseteq \bigcup_{z\in\R}B\big(\varphi(\cdot-z),\widetilde{\varepsilon}_0\big)
\]
Therefore, by making $\varepsilon_0$ smaller if necessary, we can define $x(t):=\widetilde{x}(u(t))$, and hence $x(\cdot)$ satisfies both conditions in \eqref{mod_rho_close} by construction.

\medskip
 
\textbf{Step 2:} Now we intend to prove \eqref{uniform_mod}. In fact, it is enough to notice that from the hypothesis given in \eqref{uniform_mod} together with \eqref{ift_consequences} we conclude that for any $c>0$ and any $0<\varepsilon<c\varepsilon_0$ we have \[
\left\Vert \tfrac{1}{\sqrt{c}}u(t)-\varphi(\cdot-z(t))\right\Vert_{H^1}\leq \left(\dfrac{\varepsilon}{c}\right)^2+\sup_{\vert z\vert \leq C_0\left(\frac{\varepsilon}{c}\right)^2}\Vert \varphi-\varphi(\cdot-z)\Vert_{H^1}\lesssim\varepsilon^2+C_0^{1/2}\varepsilon.
\]
\textbf{Step 3:} Our aim now is to prove \eqref{uniform_bound_param_mod}. For the sake of simplicity let us start by defining some new variables:
\[
v(t):=u-\sqrt{c}\varphi(\cdot-x(t)), \quad w:=\sqrt{c}\varphi(\cdot-x(t)) \  \hbox{ and } \ w_{n_0}:=\sqrt{c}(\rho_{n_0}*\varphi)(\cdot-x(t)).
\]
Now we recall that in view of \eqref{nov_eq_2} we know that any solution $u\in C(\R, H^1(\R))$ of the Novikov equation satisfies \[
u_t\in C(\R,L^2(\R)) \quad \hbox{and hence} \quad u\in C^1(\R,L^2(\R)).
\]
This ensures that the mapping $t\mapsto x(t)=\widetilde{x}(u(t))$ is of class $C^1$ on $\R$. Thus, by differentiating the orthogonality condition \eqref{mod_rho_close} we obtain
\[
\int v_tw_{n_0,x}=\dot{x}\int v(t,x)w_{n_0,xx}(t,x)dx=-\dot{x}\int v_x(t,x)w_{n_0,x}(t,x)dx
\]
On the other hand, by using $\varphi-\varphi''=2\delta$ we infer that $w$ satisfies the following equation
\[
w_t+(\dot{x}-c)w_x+w^2w_x=p_x\Big(w^3+\dfrac{3}{2}ww_x^2\Big)-\dfrac{1}{2}p*w_x^3.
\]
Replacing the latter equation together with $u=v+w$ into \eqref{nov_eq_2} we obtain that $v$ satisfies
\begin{align*}
v_t-(\dot{x}-c)w_x&=-(v+w)^2v_x-(v^2+2vw)w_x-\dfrac{1}{2}p*\big(v_x^3+3v_x^2w_x+3v_xw_x^2\big)
\\ & \quad +p_x*\Big(v^3+3v^2w+3vw^2+\dfrac{3}{2}v(v_x+w_x)^2+\dfrac{3}{2}v_x^2w+2v_xww_x\Big).
\end{align*}
Taking the $L^2$-inner product on the latter equation against $w_{n_0,x}$ and by using \eqref{uniform_mod} we get \[
\left\vert(\dot{x}-c)\int w_xw_{n_0,x}dx+cO(\Vert v\Vert_{H^1}) \right\vert\leq O(\Vert v\Vert_{H^1})\lesssim cC\varepsilon_0.
\]
Therefore, by using \eqref{derivative_funct_mod_proof}, considering $n_0\in\N$ sufficiently large and $\varepsilon_0$ sufficiently small so that $C\varepsilon_0\ll1$, we conclude \eqref{uniform_bound_param_mod}.

\medskip

\textbf{Step 4:} Finally, it only remains to prove \eqref{orth_cond_def} for $n_0\in\N$ large enough. In fact, it is enough to notice that  \[
\int \varphi'(x)\varphi(x-y)dx=(1-y)e^{-y}.
\]
Hence, for $n_0\in\N$ large enough we have
\[
\dfrac{d}{dy}\int \varphi(\rho_{n_0}*\varphi)'(\cdot-y)=\int \varphi'(\rho_{n_0}*\varphi')(\cdot-y)\geq \dfrac{1}{4}e^{-1/2} \quad \hbox{on} \quad \left[-\tfrac{1}{2},\tfrac{1}{2}\right].
\]
Therefore, the mapping $y\mapsto \int_\R\varphi(\rho_{n_0}*\varphi)'(\cdot-y)$ is increasing on $[-\tfrac{1}{2},\tfrac{1}{2}]$. The proof is complete. \qed

\bigskip

\textbf{Acknowledgements :} The author is very grateful to professor Luc Molinet for encouragement me in solving this problem and for many remarkably useful conversations.

\end{document}